\documentclass[a4paper, 10pt, twoside, notitlepage]{amsart}

\usepackage{amsmath,amscd}
\usepackage{amssymb}
\usepackage{amsthm}
\usepackage{comment}
\usepackage{graphicx, xcolor}

\usepackage{mathrsfs}
\usepackage[ocgcolorlinks, linkcolor=blue]{hyperref}

\usepackage{bm}
\usepackage{bbm}
\usepackage{url}

\usepackage[utf8]{inputenc}
\usepackage{mathtools,amssymb}
\usepackage{esint}
\usepackage{tikz}
\usepackage{dsfont}
\usepackage{relsize}
\usepackage{url}
\usepackage{xcolor}
\usepackage{graphicx}
\usepackage{mathrsfs}
\usepackage[shortlabels]{enumitem}
\usepackage{lineno}
\usepackage{amsmath}
\usepackage{enumitem}
\usepackage{amsthm} 
\usepackage{verbatim}
\usepackage{dsfont}
\numberwithin{equation}{section}

\allowdisplaybreaks

\mathtoolsset{showonlyrefs}

\graphicspath{{images/}}

\newtheorem{theorem}{Theorem}[section]
\newtheorem{lemma}[theorem]{Lemma}
\newtheorem{definition}[theorem]{Definition}
\newtheorem{corollary}[theorem]{Corollary}

\newtheorem{proposition}[theorem]{Proposition}
\newtheorem{question}{Question}
\newtheorem{remark}[theorem]{Remark}

\title[Global uniqueness for nonlocal diffusion]{The Calder\'on problem for a nonlocal diffusion equation with time-dependent coefficients}

\author[Y.-H. Lin]{Yi-Hsuan Lin}
\address{Department of Applied Mathematics, National Yang Ming Chiao Tung University, Hsinchu, Taiwan}
\email{yihsuanlin3@gmail.com}

\author[J. Railo]{Jesse Railo}
\address{Department of Pure Mathematics and Mathematical Statistics, University of
	Cambridge, Cambridge CB3 0WB, UK}
\email{jr891@cam.ac.uk}

\author[P. Zimmermann]{Philipp Zimmermann}
\address{Department of Mathematics, ETH Zurich, Z\"urich, Switzerland}
\email{philipp.zimmermann@math.ethz.ch}

\newcommand{\R}{{\mathbb R}}

\newcommand{\N}{{\mathbb N}}

\newcommand {\p} {\partial}

\newcommand{\LC}{\left(}
\newcommand{\RC}{\right)}
\newcommand{\wt}{\widetilde}

\newcommand{\schwartz}{\mathscr{S}}

\newcommand{\tempered}{\mathscr{S}^{\prime}}

\newcommand{\fourier}{\mathcal{F}}
\newcommand{\ifourier}{\mathcal{F}^{-1}}
\newcommand{\vev}[1]{\left\langle#1\right\rangle}

\newcommand{\distr}{\mathscr{D}^{\prime}}
\newcommand{\norm}[1]{\lVert #1 \rVert}
\newcommand{\abs}[1]{\left\lvert #1 \right\rvert}
\DeclareMathOperator{\Div}{div} 
\DeclareMathOperator{\supp}{supp} 

\newcommand{\weak}{\rightharpoonup}
\newcommand{\weakstar}{\overset{\ast}{\rightharpoonup}}

\begin{document}

	\maketitle
	\begin{abstract}
	We investigate global uniqueness for an inverse problem for a nonlocal diffusion equation on domains that are bounded in one direction. The coefficients are assumed to be unknown and isotropic on the entire space. We first show that the partial exterior \emph{Dirichlet-to-Neumann map} locally determines the diffusion coefficient in the exterior domain. In addition, we introduce a novel analysis of \emph{nonlocal Neumann derivatives} to prove an interior determination result. Interior and exterior determination yield the desired global uniqueness theorem for the Calderón problem of nonlocal diffusion equations with time-dependent coefficients. This work extends recent studies from nonlocal elliptic equations with global coefficients to their parabolic counterparts. The results hold for any spatial dimension $n\geq 1$.
		
		\medskip
		
		\noindent{\bf Keywords.} Fractional Laplacian, fractional gradient, Calderón problem, conductivity equation, Liouville reduction, nonlocal Neumann derivative, Runge approximation.
		
		\noindent{\bf Mathematics Subject Classification (2020)}: Primary 35R30; secondary 26A33, 42B37

	\end{abstract}

	\tableofcontents

	\section{Introduction}\label{sec: introduction}

	The research of inverse problems has become an active field in applied
	mathematics since Calder\'on published his pioneering work ``On
	an inverse boundary value problem'' \cite{calderon2006inverse}. Calderón asked the following question: ``Can one determine
	the electrical conductivity of a medium by using boundary measurements of voltage and current?'' This problem is referred as the \emph{Calder\'on problem} in the literature. The mathematical setup is to consider a bounded domain $\Omega\subset\mathbb{R}^{n}$ with sufficiently regular boundary $\p \Omega$, representing a conducting medium, and a positive function $\gamma(x)>0$ on $\Omega$ which is its a priori unknown conductivity. It is known that sufficiently regular conductivities are uniquelly determined by
	the information of current and voltage measurements on the boundary. In other words, $\gamma$ can be recovered when the Cauchy data 
	$\left\{u|_{\partial\Omega},\left. \gamma\frac{\partial u}{\partial\nu}\right|_{\partial\Omega}\right\}$ is given,
	where $u$ solves the conductivity equation 
	\begin{align}\label{eq:conductivity equation}
		\Div\left(\gamma\nabla u\right)=0 \text{ in }\Omega.
	\end{align}
	The Calder\'on problem was first solved by \cite{SU87} in space
	dimension $n\geq3$, where the authors demonstrated the fact that the conductivity
	can be determined uniquely by the \emph{Dirichlet-to-Neumann map} (DN map, $u|_{\partial\Omega}\mapsto\gamma\frac{\partial u}{\partial\nu}$)
	of the conductivity equation \eqref{eq:conductivity equation}. After some years, the same result has been showed in space dimension $n=2$ in \cite{Nachman1996GlobalUniqueness} and later for conductivities which are only uniformly elliptic \cite{AP06}.

	Recently, the studies of Calder\'on type inverse problems have been considered for nonlocal operators as well. A prototypical example is the inverse exterior value problem for the fractional Schr\"odinger operator $(-\Delta)^{s}+q(x)$ which was first introduced and solved in \cite{GSU20}. The main tool in solving this Calder\'on problem is based on a suitable \emph{unique continuation property} (UCP) and the closely related \emph{Runge approximation}. By applying similar ideas, one can solve several challenging problems which some still stay open in the corresponding local cases. This shows that nonlocal inverse problems take advantages from the \emph{nonlocality} of the underlying operators. For further details we refer to \cite{bhattacharyya2021inverse,CMR20,CMRU20,C20,GLX,CL2019determining,CLL2017simultaneously,CRZ2022global,cekic2020calderon,feizmohammadi2021fractional,harrach2017nonlocal-monotonicity,harrach2020monotonicity,GRSU18,GU2021calder,ghosh2021non,lin2020monotonicity,LL2020inverse,LL2022inverse,LLR2019calder,LLU2022calder,KLW2021calder,RS17,ruland2018exponential,RZ2022FracCondCounter,RZ2022LowReg,RZ2022unboundedFracCald} and the references therein. We emphasize that most of these works consider nonlocal inverse problems in which one wants to recover lower order coefficients. On the other hand, in the articles \cite{C20,GU2021calder,LLU2022calder,RZ2022LowReg,RZ2022unboundedFracCald,RZ2022FracCondCounter} the authors study nonlocal inverse problems where one is interested in determining leading order coefficients and hence they can be seen as full nonlocal analogies of the classical Calder\'on problem.
	
	\subsection{Mathematical modeling and main results}
	Let $\Omega\subset\R^n$ be an open set bounded in one direction for any $n\in \N$, and consider the initial exterior value problem of the \emph{variable coefficient nonlocal diffusion equation}
	\begin{align}\label{main eq nonlocal diff}
		\begin{cases}
			\partial_tu +\Div_s(\Theta_{\gamma}\nabla^s u)=0 & \text{ in }\Omega_T,\\
			u=f & \text{ in }(\Omega_e)_T,\\
			u(x,0)=0 &\text{ in }\Omega,
		\end{cases}
	\end{align}
	where  $\Omega_e\vcentcolon =\R^n\setminus\overline{\Omega}$ denotes the exterior of $\Omega$, $0<T<\infty$ and $0<s<\min(1,n/2)$. Throughout this work, let us assume $\gamma\in L^\infty(\R^n_T)$ is a uniformly elliptic conductivity, i.e., there exists a constant $\gamma_0>0$ such that
	\begin{align}\label{ellipticity} 
		0<\gamma_0\leq \gamma(x,t)\leq \gamma_0^{-1} \text{ for } (x,t)\in \R^n_T.
	\end{align} 
	In addition, let us denote $\Theta_{\gamma}(x,y,t)\vcentcolon =\gamma^{1/2}(x,t)\gamma^{1/2}(y,t)\mathbf{1}_{n\times n}$ for $x,y\in\R^n$ to be the conductivity matrix.
	Moreover, we always use the notation
	\[
	A_T \vcentcolon = A \times (0,T)
	\]
	to denote the space time cylinders, where $A\subset \R^n$ can be any set.

	In this work, we are interested in the determination of the conductivity $\gamma(x,t)$ in $\R^n_T$ for the nonlocal diffusion equation \eqref{main eq nonlocal diff}. Assuming the well-posedness of \eqref{main eq nonlocal diff} at moment (the proof will be given in Section \ref{sec: forward problem}), we can define the DN map via
	\begin{align}\label{DN map in intro}
		\begin{split}
			\langle \Lambda_{\gamma}f,g\rangle\vcentcolon
			=&\frac{C_{n,s}}{2}\int_0^T\int_{\R^{2n}}\gamma^{1/2}(x,t)\gamma^{1/2}(y,t)\\
			&\qquad \qquad \quad \cdot \frac{(u_f(x,t)-u_f(y,t))(g(x,t)-g(y,t))}{|x-y|^{n+2s}}\,dxdydt,
		\end{split}
	\end{align}
	for all $f,g\in C_c^{\infty}((\Omega_e)_T)$, where 
	\begin{align}\label{C_ns}
	    C_{n,s}:=\frac{4^s \Gamma(n/2+s)}{\pi^{n/2}|\Gamma(-s)|}
	\end{align}
	is a constant and $u_f$ is the unique solution of \eqref{main eq nonlocal diff}. 
	More precisely, we ask the following question:
	\begin{question}
		\label{quest}
		If we have given conductivities $\gamma_1$, $\gamma_2$ in a suitable function space such that $\Lambda_{\gamma_1}f|_{(W_2)_T}=\Lambda_{\gamma_2}f|_{(W_2)_T}$ for all $f\in C_c^{\infty}((W_1)_T)$, where $W_1,W_2\subset \Omega_e$ are given nonempty open sets, does there hold $\gamma_1=\gamma_2$ in $\R^n_T$?
	\end{question}
	In the limiting case $s=1$, this problem and its generalizations has been studied, for example, in \cite{canuto2001determining} or \cite{AliParCald}.  In these works, the authors determine the coefficients for heat equations for any spatial dimension $n\geq 2$ by using the corresponding boundary measurements, where they allow an additional uniformly elliptic coefficient $\rho$ in front of the time derivative. On the other hand, in the works \cite{GlobalTimeFrac,Time-Frac-Diff}, the inverse problem for the diffusion equation with fractional time derivative has been studied.
	
	Next, let $u_j$ be the solution of 
	\begin{align}\label{main eq nonlocal diff j=12}
		\begin{cases}
			\partial_tu_j +\Div_s(\Theta_{\gamma_j}\nabla^s u_j)=0 & \text{ in }\Omega_T,\\
			u_j=f & \text{ in }(\Omega_e)_T,\\
			u_j(x,0)=0 &\text{ in }\Omega,
		\end{cases}
	\end{align}
	and denote the exterior DN map of \eqref{main eq nonlocal diff j=12} by $\Lambda_{\gamma_j}$, for $j=1,2$. 
	Our first theorem shows that the exterior DN maps have a unique continuation property. The proof is based on a spacetime Liouville reduction, which reduces the problem to a diffusion Schr\"odinger type inverse problem. By applying the Runge approximation property for certain equations, we can prove the uniqueness of the conductivity $\gamma$. The argument however requires the Alessandrini identity as well as the use of the UCP of the fractional Laplacian twice, once in $H^s$ and once in $H^{2s,\frac{n}{2s}}$, the latter using a general UCP result in \cite{KRZ2022Biharm}. See Theorem \ref{thm:fractionalLiouvilleReduction} for a further information why earlier approaches that work well in the elliptic case lead into additional challenges in the studied parabolic case.
	
	\begin{theorem}[Global uniqueness]\label{Theorem: General formulation}
		Let $\Omega\subset \R^n$ be an open set bounded in one direction, $0<T<\infty$, $0<s<\min(1,n/2)$, $\gamma_0>0$ and $W\subset\Omega_e$ be an open set. Assume that $\gamma_j \in \Gamma_{s,\gamma_0}(\R^n_T)\cap C^{\infty}(W_T)$\footnote{The set $\Gamma_{s,\gamma_0}(\R^n_T)$ is defined by \eqref{Gamma s gamma0} in Section \ref{sec: forward problem}.} for $j=1,2$. Then 
		\begin{align}\label{same DN in thm 1}
			\left. \Lambda_{\gamma_1}f \right|_{W_T}=\left. \Lambda_{\gamma_2}f \right|_{W_T}, \text{ for any }f\in C^\infty_c(W_T),
		\end{align}
		implies that $\gamma_1=\gamma_2$ in $\R^n_T$.
	\end{theorem}

	In order to prove Theorem \ref{Theorem: General formulation}, we first need to establish an \emph{exterior determination} result. This extends the results in \cite{CRZ2022global,RZ2022LowReg} for elliptic equations and is based on a construction of special solutions to the equation \eqref{main eq nonlocal diff} whose energies can be concentrated near a fixed point in the spacetime. Furthermore, to our best knowledge, Theorem \ref{Theorem: General formulation} is the first result to recover time-dependent coefficients in the nonlocal setup.

	\begin{theorem}[Exterior determination]\label{Theorem: Exterior determination}
		Let $\Omega\subset \R^n$ be an open set bounded in one direction, $0<T<\infty$, $0<s<\min(1,n/2)$, $\gamma_0>0$ and $W\subset\Omega_e$ be an open set. Assume that $\gamma_j \in \Gamma_{s,\gamma_0}(\R^n_T)$ for $j=1,2$. Then \eqref{same DN in thm 1} implies that $\gamma_1=\gamma_2$ a.e. in $W_T$.
	\end{theorem}

	\noindent\textbf{Ideas of the proof.} 
	Let us briefly summarize the ideas of the proof of Theorem \ref{Theorem: General formulation}. We first prove the exterior uniqueness by using the exterior information from \eqref{same DN in thm 1} such that $\gamma_1 =\gamma_2$ in $W_T$. Next, consider arbitrary nonempty disjoint open subsets $W_1,W_2\subset W\subset \Omega_e$, then Theorem \ref{Theorem: Exterior determination} implies that $\gamma_1=\gamma_2$ in $(W_1\cup W_2)_T$. We next introduce the \emph{nonlocal Neumann derivative}
	\begin{align}
		\mathcal{N}_{\gamma_j}u(x,t)=C_{n,s} \int_{\Omega} \gamma_j^{1/2}(x,t)\gamma_j^{1/2}(y,t) \frac{u(x,t)-u(y,t)}{|x-y|^{n+2s}} \, dy, \quad (x,t)\in (\Omega_e)_T,
	\end{align}
	for $j=1,2$, where $C_{n,s}$ is the constant given by \eqref{C_ns}. In particular, we can prove that 
	\begin{align}\label{same nonlocal neumann}
		\langle \mathcal{N}_{\gamma_1}f,g \rangle=\langle \mathcal{N}_{\gamma_2}f,g \rangle, \text{ for any }f\in C^\infty_c((W_1)_T) \text{ and } g\in C^\infty_c((W_2)_T),
	\end{align} 
	whenever \eqref{same DN in thm 1} holds (see Lemma \ref{Lemma: old DN implies new DN}).

	Meanwhile, we introduce the spacetime Liouville reduction, which transfer the nonlocal diffusion equation \eqref{main eq nonlocal diff} into a Schr\"odinger type equation (see \eqref{eq: weak solutions of Schrodinger equation existence}). By utilizing the identity \eqref{same nonlocal neumann} and the Liouville reduction, we can derive a suitable integral identity (see Section \ref{sec: integral id}). Now, applying the Runge approximation (Proposition \ref{Prop: runge}), we can prove the interior uniqueness $\gamma_1 =\gamma_2$ in $\Omega_T$ and $(-\Delta)^s(\gamma_1^{1/2}-\gamma_2^{1/2})=0$ in $\Omega_T$. Finally, using the UCP we can conclude the proof. We want to emphasize again that our theorems hold for any spatial dimension $n\in \N$.

	\subsection{Organization of the article} We first recall preliminaries related to function spaces and nonlocal operators in Section \ref{sec:preliminaries}. In Section \ref{sec: forward problem}, we show well-posedness of the forward problem \eqref{main eq nonlocal diff} and define the exterior DN maps. We prove the exterior determination by using \eqref{same DN in thm 1} in Section \ref{sec: exterior det}. In Section \ref{sec: DN and Liouville}, we introduce the spacetime Liouville reduction, which transfer the equation \eqref{main eq nonlocal diff} into a Schr\"odinger type diffusion equation \eqref{eq: weak solutions of Schrodinger equation existence}. We also study the well-posedness for the reduced equation. In Section \ref{sec: new DN maps}, we introduce a nonlocal Neumann derivatives for both equations \eqref{main eq nonlocal diff} and \eqref{eq: weak solutions of Schrodinger equation existence}. Finally, we prove the global uniqueness in Section \ref{sec: interior det} by deriving suitable integral identities and an approximation property. In Appendix \ref{sec: Discussion of nonlocal normal derivatives and DN maps}, we discuss and explain several connections between DN maps and nonlocal Neumann derivatives.

	\section{Preliminaries}\label{sec:preliminaries}
	
	Throughout this article the space dimension $n$ is a fixed positive integer and $\Omega \subset \R^n$ is an open set. In this section, we introduce fundamental properties of function spaces and operators which will be used in our study. 
	
	\subsection{Fractional Sobolev spaces}
	\label{sec: Fractional Sobolev spaces}
	
	We denote by $\schwartz(\R^n)$ and $\tempered(\R^n)$ Schwartz functions and tempered distributions respectively. We define the Fourier transform $\fourier\colon \schwartz(\R^n)\to \schwartz(\R^n)$ by
	\begin{equation}
		\fourier f(\xi) \vcentcolon = \int_{\R^n} f(x)e^{-\mathrm{i}x \cdot \xi} \,dx,
	\end{equation}
	which is occasionally also denoted by $\hat{f}$, and $\mathrm{i}=\sqrt{-1}$. By duality it can be extended to the space of tempered distributions and will again be denoted by $\fourier u = \hat{u}$, where $u \in \tempered(\R^n)$, and we denote the inverse Fourier transform by $\ifourier$. Next recall that the fractional Laplacian of order $a\geq 0$ as a Fourier multiplier
	\begin{equation}\label{eq:fracLapFourDef}
		(-\Delta)^{a/2} u = \ifourier\LC \abs{\xi}^{a}\hat{u}(\xi)\RC,
	\end{equation}
	for $u \in \tempered(\R^n)$ whenever the right hand side is well-defined. Given $a\geq 0$, the $L^2$-based fractional Sobolev space $H^{a}(\R^{n})\vcentcolon =W^{a,2}(\mathbb{R}^{n})$ is  given by  
	\begin{equation}\notag
		\|u\|^2_{H^{a}(\mathbb{R}^{n})}= \|u\|_{L^{2}(\mathbb{R}^{n})}^{2}+\|(-\Delta)^{a/2}u\|_{L^{2}(\mathbb{R}^{n})}^{2}.
	\end{equation}
	In addition, the Parseval identity implies that the seminorm $\|(-\Delta)^{a/2}u\|_{L^{2}(\mathbb{R}^{n})}$
	can be expressed as 
	\[
	\|(-\Delta)^{a/2}u\|_{L^{2}(\mathbb{R}^{n})}=\langle (-\Delta)^{a}u,u\rangle_{L^2(\mathbb{R}^{n})}^{1/2}.
	\]
	By duality one extends the spaces $H^a(\R^n)$ to the range $a<0$.
	If $\Omega\subset \R^n$ is an open set and $a\in\mathbb{R}$,
	then the fractional Sobolev spaces are defined by  
	\begin{align*}
		H^{a}(\Omega) & \vcentcolon =\left\{\,u|_{\Omega}\,;\, u\in H^{a}(\mathbb{R}^{n})\right\},\\
		\widetilde{H}^{a}(\Omega) & \vcentcolon =\text{closure of \ensuremath{C_{c}^{\infty}(\Omega)} in \ensuremath{H^{a}(\mathbb{R}^{n})}}.
	\end{align*}
	Meanwhile, $H^{a}(\Omega)$ is a Banach space with respect to the quotient norm
	\[
	\|u\|_{H^{a}(\Omega)}\vcentcolon =\inf\left\{ \|U\|_{H^{a}(\mathbb{R}^{n})}\,;\,U\in H^{a}(\mathbb{R}^{n})\mbox{ and }U|_{\Omega}=u\right\} .
	\]
	
	\subsection{Bessel potential spaces}
	
	Next, we introduce the Bessel potential spaces $H^{s,p}(\R^n)$ and two local variants of them, namely $\widetilde{H}^{s,p}(\Omega)$ and $H^{s,p}(\Omega)$. The Bessel potential of order $s \in \R$ is the Fourier multiplier $\vev{D}^s\colon \tempered(\R^n) \to \tempered(\R^n)$ given by
	\begin{equation}
		\label{eq: Bessel pot}
		\vev{D}^s u \vcentcolon = \ifourier(\vev{\xi}^s\hat{u}),
	\end{equation}
	where $\langle \xi\rangle=\sqrt{1+|\xi|^2}$ is the Japanese bracket. Now for any $1 \leq p < \infty$ and $s \in \R$ the Bessel potential spaces $H^{s,p}(\R^n)$ are defined by 
	\begin{equation}\label{eq: Bessel pot spaces}
		H^{s,p}(\R^n) \vcentcolon = \left\{\, u \in \tempered(\R^n)\,;\, \vev{D}^su \in L^p(\R^n)\,\right\}
	\end{equation}
	and they are equipped with the norm $\norm{u}_{H^{s,p}(\R^n)} \vcentcolon = \norm{\vev{D}^su}_{L^p(\R^n)}$. The local Bessel potential spaces $\widetilde{H}^{s,p}(\Omega)$ are now defined as the closure of $C_c^{\infty}(\Omega)$ in $H^{s,p}(\R^n)$ and endowed with the norm inherited from $H^{s,p}(\R^n)$. Moreover, we denote by $H^{s,p}(\Omega)$ the space of restrictions from elements in $H^{s,p}(\R^n)$ to $\Omega$ and endow it with the related quotient norm
	\[
	\norm{u}_{H^{s,p}(\Omega)} \vcentcolon = \inf\left\{\,\norm{U}_{H^{s,p}(\R^n)}\,;\, U \in H^{s,p}(\R^n), U|_\Omega = u\,\right\}.
	\]
	We have that $(\widetilde{H}^{s,p}(\Omega))^* = H^{-s,p'}(\Omega)$ and 
	$\widetilde{H}^{s,p}(\Omega) = (H^{-s,p'}(\Omega))^*$ for every $1 < p < \infty$ and $s \in \R$. As usual, when $p=2$, then we drop the index $p$ in the above notations and see that they are isomorphic to the spaces introduced in Section~\ref{sec: Fractional Sobolev spaces}.

	\subsection{Some properties of nonlocal operators}
	
	It is known that the fractional Laplacian induces a bounded linear map $(-\Delta)^{s/2}\colon H^{t,p}(\R^n) \to H^{t-s,p}(\R^n)$ for every $1 \leq p < \infty$, $s \geq0$ and $t \in \R$. Next, we introduce a special class of unbounded open sets which have a fractional Poincar\'e inequality:
	
	\begin{definition} \label{def:bounded1dir}
		\begin{enumerate}[(i)]
			\item\label{item 1 bounded one dir} We say that an open set $\Omega_\infty \subset\R^n$ of the form $\Omega_\infty=\R^{n-k}\times \omega$, where $n\geq k>0$ and $\omega \subset \R^k$ is a bounded open set, is a \emph{cylindrical domain}.
			\item\label{item 2 bounded one dir} We say that an open set $\Omega \subset \R^n$ is \emph{bounded in one direction} if there exists a cylindrical domain $\Omega_\infty \subset \R^n$ and a rigid Euclidean motion $A(x) = Lx + x_0$, where $L$ is a linear isometry and $x_0 \in \R^n$, such that
			$\Omega \subset A\Omega_\infty$.
		\end{enumerate}
	\end{definition}
	
	\begin{proposition}[{Poincar\'e inequality (cf.~\cite[Theorem~2.2]{RZ2022unboundedFracCald})}]\label{thm:PoincUnboundedDoms} Let $\Omega\subset\R^n$ be an open set that is bounded in one direction. Suppose that $2 \leq p < \infty$ and $0\leq s\leq t < \infty$, or $1 < p < 2$, $1 \leq t < \infty$ and $0 \leq s \leq t$. Then there exists $C(n,p,s,t,\Omega)>0$ such that
		\begin{equation}
			\label{eq: poincare on L1}
			\|(-\Delta)^{s/2}u\|_{L^p(\R^n)}\leq C\|(-\Delta)^{t/2}u\|_{L^p(\R^n)}
		\end{equation}
		for all $u\in \widetilde{H}^{t,p}(\Omega)$.
	\end{proposition}
	
	For the rest of this article we fix $s\in(0,1)$. The fractional gradient of order $s$ is the bounded linear operator $\nabla^s\colon H^s(\R^n)\to L^2(\R^{2n};\R^n)$ given by (see ~\cite{C20,DGLZ12,RZ2022unboundedFracCald})
	\[
	\nabla^su(x,y)\vcentcolon =\sqrt{\frac{C_{n,s}}{2}}\frac{u(x)-u(y)}{|x-y|^{n/2+s+1}}(x-y),
	\]
	which satisfies
	\begin{equation}
		\label{eq: bound on fractional gradient}
		\|\nabla^su\|_{L^2(\R^{2n})}=\|(-\Delta)^{s/2}u\|_{L^2(\R^n)}\leq \|u\|_{H^s(\R^n)}
	\end{equation}
	for all $u\in H^s(\R^n)$, where $C_{n,s}$ is the constant given by \eqref{C_ns}. The adjoint of $\nabla^s$ is called fractional divergence of order $s$ and denoted by $\Div_s$. More concretely, the fractional divergence of order $s$ is the bounded linear operator 
	\[
	\Div_s\colon L^2(\R^{2n};\R^n)\to H^{-s}(\R^n)
	\] 
	satisfying
	\[
	\langle \Div_su,v\rangle_{H^{-s}(\R^n)\times H^s(\R^n)}=\langle u,\nabla^sv\rangle_{L^2(\R^{2n})}
	\]
	for all $u\in L^2(\R^{2n};\R^n),v\in H^s(\R^n)$. One can show that (see ~\cite[Section 8]{RZ2022unboundedFracCald})
	\[
	\|\Div_s(u)\|_{H^{-s}(\R^n)}\leq \|u\|_{L^2(\R^{2n})}
	\]
	for all $u\in L^2(\R^{2n};\R^n)$, and also 
	$$
	(-\Delta)^su=\Div_s(\nabla^su)
	$$ 
	weakly for all $u\in H^s(\R^n)$ (see ~\cite[Lemma 2.1]{C20}). 
	
	\subsection{Bochner spaces}
	\label{subsec: Bochner spaces}
	
	Next, we introduce some standard function spaces for time-dependent PDEs adapted to the nonlocal setting considered in this article. Let $X$ be a Banach space and $(a,b)\subset\R$. Then we let $L^p(a,b\,;X)$ ($1\leq p<\infty$) stand for the space of measurable functions $u\colon (a,b)\to X$ such that 
	\begin{equation}
		\label{eq: Bochner spaces}
		\|u\|_{L^p(a,b\,;X)}\vcentcolon = \left(\int_a^b\|u(t)\|_{X}^p\,dt\right)^{1/p}<\infty
	\end{equation} 
	and $L^{\infty}(a,b\,;X)$ the space of measurable functions $u\colon (a,b)\to X$ such that 
	\begin{equation}
		\label{eq: infty spaces}
		\|u\|_{L^{\infty}(a,b\,;X)}\vcentcolon = \inf\{\,M\,;\,\|u(t)\|_X\leq M\,\text{a.e.}\,\}<\infty.
	\end{equation}
	As usual, we say that $u\in L^p_{loc}(a,b\,;X)$ if $\chi_K u\in L^p(a,b\,;X)$ for any compact set $K\subset (a,b)$, where $\chi_A$ denotes the characteristic function of the set $A$.

	Moreover, if $u\in L^1_{loc}(a,b\,;X)$ and $X$ is a space of functions over an open set $\Omega\subset\R^n$, as $L^p(\Omega)$, then $u$ is identified with a function $u(x,t)$ and $u(t)$ denotes the function $\Omega\ni x\mapsto u(x,t)$ for almost all $t$. This is justified from the fact, that any $u\in L^q(a,b\,;L^p(\Omega))$ with $1\leq q,p<\infty$ can be seen as a measurable function $u\colon \Omega\times (a,b)\to \R$ such that the norm $\|u\|_{L^q(a,b\,;L^p(\Omega))}$, as defined in \eqref{eq: Bochner spaces}, is finite. Clearly, a similar statement holds for the spaces $L^q(a,b\,;H^{s,p}(\R^n))$ and their local versions. Furthermore, the distributional derivative $\frac{du}{dt}\in \distr((a,b)\,;X)$ is identified with the derivative $\partial_tu\in \distr(\Omega\times (a,b))$ as long as it is well-defined. Here $ \distr((a,b)\,;X)$ stands for all continuous linear operators from $C_c^{\infty}((a,b))$ to $X$. Given two Banach spaces $X,Y$ such that $X\hookrightarrow Y$, we say $u\in L^2(a,b\,;X)$ has a (weak) time derivative $u'\vcentcolon =\frac{du}{dt}$ in $L^2(a,b\,;Y)$ if there exists $v\in L^2(a,b\,;Y)$ such that
	\begin{equation}
		\label{eq: weak time derivatives}
		\langle u',\eta\rangle\vcentcolon =-\int_a^bu(t)\eta'(t)\,dt=\int_a^bv(t)\eta(t)\,dt
	\end{equation}
	for $\eta\in C_c^{\infty}((a,b))$ (cf.~\cite{DautrayLionsVol5}). 
	
	\section{The forward problem of nonlocal diffusion equation}
	\label{sec: forward problem}
	
	In this section, we study the well-posedness of the initial exterior problem \eqref{main eq nonlocal diff} with possibly nonzero initial condition $u_0$ and the properties of the related DN maps.
	We start by setting up the relevant bilinear forms and then define the notion of solutions used throughout this article, which is in parallel to the theory developed for second order parabolic equations (see e.g. \cite{LadyzhenskajaParabolic,LadyzhenskajaBVP}).

	\begin{definition}[Definition of bilinear forms and conductivity matrix]\label{def: bilinear forms conductivity eq}
		Let $\Omega\subset\R^n$ be an open set, $0<s<\min(1,n/2)$, $\gamma\in L^{\infty}(\R^n_T)$. Then we define the conductivity matrix associated to $\gamma$ by
		\begin{equation}
			\label{eq: conductivity matrix}
			\Theta_{\gamma}\colon \R^{2n}\times (0,T)\to \R^{n\times n},\quad \Theta_{\gamma}(x,y,t)\vcentcolon =\gamma^{1/2}(x,t)\gamma^{1/2}(y,t)\mathbf{1}_{n\times n}
		\end{equation}
		for $x,y\in\R^n$, $0<t<T$ and the following time-dependent bilinear form for the fractional conductivity operator
		\begin{equation}
			\label{eq: conductivity bilinear form}
			\begin{split}
				&B_{\gamma}(t\,;\cdot,\cdot)\colon H^s(\R^n)\times H^s(\R^n)\to \R,\\
				&B_{\gamma}(t\,;u,v)\vcentcolon =\int_{\R^{2n}}\Theta_{\gamma}(t)\nabla^su\cdot\nabla^sv\,dxdy.
			\end{split}
		\end{equation}
		Moreover, if $m_{\gamma}\in L^{\infty}(0,T\,;H^{2s,\frac{n}{2s}}(\R^n))$, where $m_\gamma$ denotes the \emph{background deviation} 
		\begin{align}
			\label{m_gamma}
			m_{\gamma}\vcentcolon = \gamma^{1/2}-1 \text{ in }\R^n_T,
		\end{align}
		then we let $q_{\gamma}(t)\colon H^s(\R^n)\times H^s(\R^n)\to\R$ be defined by
		\begin{equation}
			\label{eq: def potential}
			\begin{split}
				\langle q_{\gamma}(t)u,v\rangle &\vcentcolon = -\left\langle\frac{(-\Delta)^sm_{\gamma}}{\gamma^{1/2}}u,v\right\rangle_{L^2(\R^n)},\\
			\end{split}
		\end{equation}
		for $u,v\in H^s(\R^n)$. In this case we define the time-dependent bilinear form for the related fractional Schr\"odinger operator with potential $q_{\gamma}$:
		\begin{equation}
			\label{eq: Schroedinger bilinear form}
			\begin{split}
				&B_{q_{\gamma}}(t\,;\cdot,\cdot)\colon H^s(\R^n)\times H^s(\R^n)\to \R,\\
				&B_{q_{\gamma}}(t\,;u,v)\vcentcolon =\int_{\R^n}(-\Delta)^{s/2}u\,(-\Delta)^{s/2}v\,dx+\int_{\R^n}q_{\gamma}(t)uv\,dx
			\end{split}
		\end{equation}
		for all $u,v\in H^s(\R^n)$.
	\end{definition}
	\begin{remark}
		If no confusion can arise we will drop the subscript $\gamma$ in the definition for the conductivity matrix $\Theta_{\gamma}(t)$. Moreover, the boundedness and coercivity of these bilinear forms is established in Lemma~\ref{lem: continuity}.
	\end{remark}
	
	\begin{lemma}
		\label{lem: continuity}
		Let $0<s<\min(1,n/2)$ and suppose $\gamma=\gamma(x,t)\in L^{\infty}(\R^n_T)$ is uniformly elliptic satisfying \eqref{ellipticity}. If the background deviation $m_{\gamma}$ of $\gamma$ satisfies $m_{\gamma}\in L^{\infty}(0,T\,;H^{2s,\frac{n}{2s}}(\R^n))$, then there exists $C>0$ such that
		\begin{equation}
			\label{eq: uniform estimate conductivity equation}
			\left|B_{\gamma}(t\,;u,v)\right|\leq \|\gamma\|_{L^{\infty}(\R^n_T)}\|u\|_{H^s(\R^n)}\|v\|_{H^s(\R^n)}
		\end{equation}
		and
		\begin{align}\label{eq: uniform estimate fractional Schrodinger equation case 1}
			\left|B_{q_{\gamma}}(t\,;u,v)\right|\leq C\|u\|_{H^s(\R^n)}\|v\|_{H^s(\R^n)}.
		\end{align}
		
		Moreover, if $\Omega\subset\R^n$ is an open set which is bounded in one direction, then the bilinear form $B_{\gamma}(t\,;\cdot,\cdot)$ is uniformly coercive over $\widetilde{H}^s(\Omega)$, that is, there exists $c>0$ such that
		\begin{equation}
			\label{eq: uniform coercivity}
			B_{\gamma}(t\,;u,u)\geq c\|u\|_{H^s(\R^n)}^2
		\end{equation}
		for all $u\in H^s(\R^n)$ and a.e. $0<t<T$.
	\end{lemma}
	
	
	\begin{proof}
		Throughout the proof we will write $m$ and $q$ instead of $m_{\gamma},q_{\gamma}$. The estimate \eqref{eq: uniform estimate conductivity equation} follows immediately from \eqref{eq: bound on fractional gradient}. Next note that by \cite[Lemma~8.3]{RZ2022unboundedFracCald}, the uniform ellipticity of $\gamma$ and the boundedness of the fractional Laplacian there holds
		\[
		\begin{split}
			\left|B_q(t\,;u,v)\right|&\leq C\LC 1+\|q(t)\|_{L^{\frac{n}{2s}}(\R^n)}\RC \|u\|_{H^s(\R^n)}\|v\|_{H^s(\R^n)}\\
			&\leq C\LC 1+\|m(t)\|_{H^{2s,\frac{n}{2s}}(\R^n)}\RC \|u\|_{H^s(\R^n)}\|v\|_{H^s(\R^n)}\\
			&\leq C\gamma_0^{1/2}\LC 1+\|m\|_{L^{\infty}(0,T\,;H^{2s,\frac{n}{2s}}(\R^n))}\RC\|u\|_{H^s(\R^n)}\|v\|_{H^s(\R^n)}.
		\end{split}
		\]
		The uniform coercivity estimate \eqref{eq: uniform coercivity} of $B_{\gamma}(t\,;\cdot,\cdot)$ follows by the uniform ellipticity of $\gamma$, \eqref{eq: bound on fractional gradient} and the Poincar\'e inequality, cf. Proposition~\ref{thm:PoincUnboundedDoms}.
	\end{proof}
	
	\begin{definition}[Weak solutions]
		Let $\Omega\subset \R^n$ be an open set, $0<T<\infty$, $0<s<1$ and assume that $\gamma\in L^{\infty}(\R^n_T)$ is uniformly elliptic. Let $u_0\in L^2(\Omega)$, $f\in L^2(0,T\,;H^s(\R^n))$ and $F\in L^2(0,T\,;H^{-s}(\Omega))$. 
		\begin{enumerate}[(i)]
			\item\label{item: nonlocal heat equation} We say that $u\in  L^2(0,T\,;H^s(\R^n))$ solves the nonlocal diffusion equation 
			\begin{equation}
				\label{eq: solutions of PDE}
				\partial_tu+\Div_s(\Theta_{\gamma}\nabla^s u)= F\quad\text{in}\quad\Omega_T,
			\end{equation}
			if the equation is satisfied in the sense of distributions, that is, there holds
			\begin{equation}
				\label{eq: distributional solutions}
				\mathbf{B}_{\gamma}(u,\varphi)\vcentcolon = -\int_{\Omega_T}u\partial_t\varphi\,dxdt+\int_0^T B_{\gamma}(t\,;u,\varphi)\,dt=\langle F,\varphi\rangle
			\end{equation}
			for all $\varphi \in C_c^{\infty}(\Omega_T)$, where $\langle\cdot,\cdot\rangle$ denotes the natural duality pairing.
			\item\label{item: initial value problem nonlocal heat equation} We say that $u\in L^{\infty}(0,T\,;L^2(\Omega))\cap L^2(0,T\,;H^s(\R^n))$ solves
			\begin{align}
				\label{eq: weak solutions of conductivity equation}
				\begin{cases}
					\partial_tu+\Div_s(\Theta_{\gamma}\nabla^s u)= F   & \text{ in }\Omega_T,\\
					u= f  &\text{ in } (\Omega_e)_T,\\
					u(0)=u_0 & \text{ in } \Omega,
				\end{cases}
			\end{align}
			if the exterior value $f$ is attained in the sense $u-f\in L^2(0,T\,;\widetilde{H}^s(\Omega))$ and there holds
			\begin{equation}
				\label{eq: solutions exterior value problem nonlocal heat eq}
				\mathbf{B}_{\gamma}(u,\varphi)=\langle F,\varphi\rangle+\int_{\Omega} u_0(x)\varphi(x,0)\,dx
			\end{equation}
			for all $\varphi\in C_c^{\infty}(\Omega\times [0,T))$.
		\end{enumerate}
	\end{definition}
	
	
	\begin{remark}
		\label{remark: why initial condition in L Omega}
		Let us briefly explain why we prescribed the initial condition in $L^2(\Omega)$ and not in $L^2(\R^n)$. One knows that there holds $(\widetilde{H}^s(\Omega))^*=H^{-s}(\Omega)$ for any $s\in\R$ and any open set $\Omega\subset \R^n$. On the other hand, by density of $C^{\infty}_c(\Omega_T)$ in $L^2(0,T\,;\widetilde{H}^s(\Omega))$ the equation \eqref{eq: solutions of PDE} implies that $\partial_tu$ can be identified with an element in $L^2(0,T\,;H^{-s}(\Omega))$. By the trace theorem \cite[Chapter~XVIII, Section~1.2, Theorem~1]{DautrayLionsVol5} this implies $u\in C([0,T];L^2(\Omega))$. Thus, $u\in L^{\infty}(0,T\,;L^2(\Omega))\cap L^2(0,T\,;H^s(\R^n))$ is a solution to \eqref{eq: weak solutions of conductivity equation} if and only if $u\in L^2(0,T;H^s(\R^n))$ with $\partial_tu\in L^2(0,T\,;H^{-s}(\Omega))$ solves \eqref{eq: solutions of PDE} in the sense of distributions, $u-f\in L^2(0,T\,;\widetilde{H}^s(\Omega))$ and there holds $u(0)=u_0$ in the sense of traces. 
		By approximation one sees that $u\in L^{\infty}(0,T\,;L^2(\Omega))\cap L^2(0,T\,;H^s(\R^n))$ is a solution of \eqref{eq: solutions of PDE} if and only if $u\in L^2(0,T\,;H^s(\R^n))$ with $\partial_t u\in L^2(0,T\,; H^{-s}(\Omega))$ satisfies
		\begin{equation}
			\label{eq: simplified definition 2}
			\langle \partial_t u,\varphi\rangle_{H^{-s}(\Omega)\times \widetilde{H}^s(\Omega)}+B_{\gamma}(t\,;u,\varphi)=\langle F(t),\varphi\rangle_{H^{-s}(\Omega)\times \widetilde{H}^s(\Omega)}
		\end{equation}
		for all $\varphi \in \widetilde{H}^s(\Omega)$ in the sense of distributions on $(0,T)$, $u(0)=u_0$ and $u-f\in L^2(0,T\,;\widetilde{H}^s(\Omega))$.
	\end{remark}
	
	\begin{theorem}[Well-posedness of the forward problem]
		\label{thm: well-posedness of forward problem}
		Let $\Omega\subset \R^n$ be an open set bounded in one direction, $0<T<\infty$, $0<s<\min(1,n/2)$ and assume that $\gamma\in L^{\infty}(\R^n_T)$ is uniformly elliptic. Assume that $F\in L^2(0,T\,;H^{-s}(\Omega))$, $f\in L^2(0,T\,;H^s(\R^n))$ with $\partial_tf\in L^2(0,T\,;H^{-s}(\R^n))$ and $u_0\in L^2(\Omega)$.
		\begin{enumerate}[(i)]
			\item Then there exists a unique solution $u\in L^{\infty}(0,T\,;L^2(\Omega))\cap L^2(0,T\,;H^s(\R^n))$ and $\p_t u \in L^2(0,T\,;H^{-s}(\Omega))$ of 
			\begin{align}
				\label{eq: weak solutions of conductivity equation existence}
				\begin{cases}
					\partial_tu+\Div_s \LC \Theta_{\gamma}\nabla^s u \RC= F &\text{ in }\Omega_T,\\
					u= f   &\text{ in }(\Omega_e)_T,\\
					u(0)=u_0  & \text{ in } \Omega
				\end{cases}
			\end{align}
			satisfying the energy estimate
			\begin{equation}
				\label{eq: energy estimate}
				\begin{split}
					&\|u-f\|_{L^{\infty}(0,T\,;L^2(\Omega))}^2+\|u-f\|_{L^2(0,T\,;H^s(\R^n))}^2+\norm{\p_t (u-f)}_{L^2(0,T\,;H^{-s}(\Omega))}^2\\
					\leq & C \left( \|u_0\|_{L^2(\Omega)}^2+\|f(0)\|_{L^2(\Omega)}^2+\|F\|^2_{L^2(0,T\,;H^{-s}(\Omega))}+\|\partial_tf\|^2_{L^2(0,T\,;H^{-s}(\Omega))}\right. \\
					& \left. +\|\Div_s(\Theta_{\gamma}\nabla^sf)\|^2_{L^2(0,T\,;H^{-s}(\Omega))}\right),
				\end{split}
			\end{equation}
			for  some constant $C>0$ independent of $F,f$ and $u_0$.
			\item If additionally the conductivity $\gamma$ satisfies $m_{\gamma}\in L^{\infty}(0,T\,;H^{4s,\frac{n}{2s}}(\R^n))$ with $\partial_t\gamma\in L^{\infty}(\R^n_T)$, $F\in L^2(\Omega_T)$, $f\in H^1(0,T\,;L^2(\R^n))\cap L^2(0,T\,;H^{2s}(\R^n))$ and $u_0\in H^s(\R^n)$ such that $u_0-f(0)\in \widetilde{H}^s(\Omega)$ then the unique solution $u$ to \eqref{eq: weak solutions of conductivity equation existence} satisfies $u\in L^{\infty}(0,T\,;H^s(\R^n))$, $\partial_tu\in L^2(\Omega_T)$ and
			\begin{equation}
				\label{eq: bound time derivate}
				\begin{split}
					&\|\partial_t(u-f)\|^2_{L^2(\Omega_T)}+\|u-f\|^2_{L^{\infty}(0,T\,;H^s(\R^n))}\\
					\leq & C\left(\|u_0\|_{H^s(\R^n)}^2+\|F\|^2_{L^2(\Omega_T)} \right.\\
					&\quad \left. +\|f(0)\|^2_{H^s(\R^n)}+\|\partial_t f\|_{L^2(\Omega_T)}^2+\|f\|_{L^2(0,T\,;H^{2s}(\R^n))}^2\right),
				\end{split}
			\end{equation}
			for some $C>0$ independent of the data $F,f$ and $u_0$.
		\end{enumerate}
	\end{theorem}
	
	\begin{remark}
		\label{rem: notation}
		If $F=u_0=0$, let us denote the unique solution of \eqref{eq: weak solutions of conductivity equation existence} by $u_f$ for simplicity.
	\end{remark}

	\begin{proof}[Proof of Theorem \ref{thm: well-posedness of forward problem}]
		(i): By the regularity assumptions on the exterior value $f$ and the trace theorem \cite[Chapter~XVIII, Section~1.2, Theorem~2]{DautrayLionsVol5}, we see that $u\in L^{\infty}(0,T\,;L^2(\Omega))\cap L^2(0,T\,;H^s(\R^n))$ is a solution of \eqref{eq: weak solutions of conductivity equation existence} if and only if $\widetilde{u}\vcentcolon =u-f\in L^{\infty}(0,T\,;L^2(\Omega))\cap L^2(0,T\,;H^s(\R^n))$ solves the homogeneous problem 
		\begin{align}
			\label{eq: weak solutions of homogeneous conductivity equation existence}
			\begin{cases}
				\partial_t\widetilde{u}+\Div_s \LC \Theta_{\gamma}\nabla^s \widetilde{u}\RC= \widetilde{F} &\text{ in }\Omega_T,\\
				\widetilde{u}= 0  &\text{ in } (\Omega_e)_T,\\
				\widetilde{u}(0)=\widetilde{u}_0 &\text{ in }\Omega,
			\end{cases}
		\end{align}
		with $\widetilde{u}_0\vcentcolon = u_0-f(0)\in L^2(\Omega)$ and
		\[
		\widetilde{F}\vcentcolon = F-\partial_tf-\Div_s(\Theta_{\gamma}\nabla^sf)\in L^2(0,T\,;H^{-s}(\Omega)).
		\]
		Note that by Remark~\ref{remark: why initial condition in L Omega} this means that $\widetilde{u}\in L^{\infty}(0,T\,;L^2(\Omega))\cap L^2(0,T\,;\widetilde{H}^s(\Omega))$ satisfies
		\begin{equation}
			\label{eq: simplified definition 3}
			\langle \partial_t \widetilde{u},\varphi\rangle_{H^{-s}(\Omega)\times \widetilde{H}^s(\Omega)}+B_{\gamma}(t\,;\widetilde{u},\varphi)=\langle \widetilde{F}(t),\varphi\rangle_{H^{-s}(\Omega)\times \widetilde{H}^s(\Omega)}
		\end{equation}
		for all $\varphi \in \widetilde{H}^s(\Omega)$ in the sense of distributions on $(0,T)$, $\widetilde{u}(0)=\widetilde{u}_0$. Now one can construct the solution $\widetilde{u}$ by the classical Galerkin approximation. In fact, using Lemma~\ref{lem: continuity} we deduce from \cite[Chapter~XVIII, Section~3.1-3.2, Theorem~1 and 2]{DautrayLionsVol5} that this problem has a unique solution $\widetilde{u}\in L^2(0,T\,;\widetilde{H}^s(\Omega))$ with $\partial_t \widetilde{u}\in L^2(0,T\,;H^{-s}(\Omega))$. Hence, our solution to the original problem is $u=\widetilde{u}+f$.

		Next we prove that this solution is the unique solution to \eqref{eq: weak solutions of conductivity equation existence}. Assume there are two solutions $u,v\in L^{\infty}(0,T\,;L^2(\Omega))\cap L^2(0,T\,;H^s(\R^n))$ to \eqref{eq: weak solutions of conductivity equation existence} then $w\vcentcolon = u-v$ solves  
		\begin{align}
			\label{eq: weak solutions uniqueness}
			\begin{cases}
				\partial_tw+\Div_s \LC \Theta_{\gamma}\nabla^s w \RC= 0 &\text{ in }\Omega_T,\\
				w= 0   &\text{ in }(\Omega_e)_T,\\
				w(0)=0  & \text{ in } \Omega.
			\end{cases}
		\end{align}
		But this means by approximation and integration by parts that there holds
		\[
		\int_0^T\langle \partial_t w,\varphi\rangle \eta\,dt+\int_0^TB_{\gamma}(w,\varphi)\,\eta\,dt=0
		\]
		for all $\varphi\in \widetilde{H}^s(\Omega)$, $\eta\in C_c^{\infty}((0,T))$ and hence
		\[
		\langle \partial_t w,\varphi\rangle +B_{\gamma}(w,\varphi)=0
		\]
		for a.e. $t\in (0,T)$. Hence, replacing $\varphi$ by $w(t)\in \widetilde{H}^s(\Omega)$ and integrating the resulting equation over $(0,T)$ gives
		\[
		\frac{\|w(T)\|_{L^2(\Omega)}^2}{2}+\int_0^{T}B_{\gamma}(w,w)\,dt=0.
		\]
		Here we used $w(0)=0$ in $\Omega$ and the integration by parts formula in Banach spaces. Using the uniform ellipticity of $\gamma$ and Poincar\'e's inequality it follows that $w=0$ and therefore $u=v$ in $\R^n_T$.\\ 
		Next we show the energy estimate \eqref{eq: energy estimate}. By \cite[eq.~(3.70)]{DautrayLionsVol5} there holds
		\begin{equation}\label{eq: energy estimate 1}
			\begin{split}
				&\frac{\|(u-f)(t)\|_{L^2(\Omega)}^2}{2} +\int_0^t\int_{\R^{2n}}\Theta_{\gamma}\nabla^s(u-f)(\tau)\cdot\nabla^s(u-f)(\tau)\,dxdyd\tau\\
				=&\frac{\|u_0-f(0)\|_{L^2(\Omega)}^2}{2}+\int_0^t\langle \widetilde{F}(\tau),(u-f)(\tau)\rangle_{H^{-s}(\Omega)\times \widetilde{H}^s(\Omega)}\,d\tau 
			\end{split}
		\end{equation}
		for all $t\in (0,T)$. The right hand side can be estimated as
		\begin{equation}
			\label{eq: estimate RHS}
			\begin{split}
				&\frac{\|u_0-f(0)\|_{L^2(\Omega)}^2}{2}+\int_0^t\langle \widetilde{F}(\tau ),(u-f)(\tau )\rangle_{H^{-s}(\Omega)\times \widetilde{H}^s(\Omega)}\,d\tau \\
				\leq &C(\|u_0\|_{L^2(\Omega)}^2+\|f(0)\|_{L^2(\Omega)}^2)+\|\widetilde{F}\|_{L^2(0,T\,;H^{-s}(\Omega))}\|u-f\|_{L^2(0,T\,;\widetilde{H}^s(\Omega))}\\
				\leq &C(\|u_0\|_{L^2(\Omega)}^2+\|f(0)\|_{L^2(\Omega)}^2)+(\|F\|_{L^2(0,T\,;H^{-s}(\Omega))}+\|\partial_tf\|_{L^2(0,T\,;H^{-s}(\Omega))}\\
				& +\|\Div_s(\Theta_{\gamma}\nabla^sf)\|_{L^2(0,T\,;H^{-s}(\Omega))})\|u-f\|_{L^2(0,T\,;H^s(\R^n))}.
			\end{split}
		\end{equation}

		On the other hand using the uniform ellipticity of $\gamma$ and the fractional Poincar\'e inequality, the left hand side of \eqref{eq: energy estimate} can be bounded from below by
		\begin{equation}
			\label{eq: lower bound LHS}
			\begin{split}
				&\frac{\|(u-f)(t)\|_{L^2(\Omega)}^2}{2} +\int_0^t\int_{\R^{2n}}\Theta_{\gamma}\nabla^s(u-f)(\tau )\cdot\nabla^s(u-f)(\tau )\,dxdyd\tau \\
				\geq &\frac{\|(u-f)(t)\|_{L^2(\Omega)}^2}{2}+c\|\nabla^s(u-f)\|_{L^2(0,t\,;L^2(\R^{2n}))}^2,\\
				\geq & \frac{\|(u-f)(t)\|_{L^2(\Omega)}^2}{2}+c\|(-\Delta)^{s/2}(u-f)\|_{L^2(0,t\,;L^2(\R^n))}^2\\
				\geq &c \LC \|(u-f)(t)\|_{L^2(\Omega)}^2+\|u-f\|_{L^2(0,t\,;H^s(\R^n))}^2\RC.
			\end{split}
		\end{equation}
		Hence, combining \eqref{eq: estimate RHS} and \eqref{eq: lower bound LHS} we deduce
		\[
		\begin{split}
			&\|u-f\|_{L^{\infty}(0,T\,;L^2(\Omega))}^2+\|u-f\|_{L^2(0,T\,;H^s(\R^n))}^2\\
			\leq& C \left( \|u_0\|_{L^2(\Omega)}^2+\|f(0)\|_{L^2(\Omega)}^2\right)+C\left(\|F\|_{L^2(0,T\,;H^{-s}(\Omega))}+\|\partial_tf\|_{L^2(0,T\,;H^{-s}(\Omega))}\right.\\
			&\qquad \left. +\|\Div_s(\Theta_{\gamma}\nabla^sf)\|_{L^2(0,T\,;H^{-s}(\Omega))}\right)\|u-f\|_{L^2(0,T\,;H^s(\R^n))}.
		\end{split}
		\]
		Next, recall that for all $\epsilon>0$ and $a,b\in\R$ there holds the estimate $ab\leq \epsilon a^2+C_{\epsilon}b^2$, where $C_{\epsilon}>0$. Hence, after absorbing the term $\epsilon \|u-f\|^2_{L^2(0,T\,;H^s(\R^n))}$ on the left hand side we obtain
		\[
		\begin{split}
			&\|u-f\|_{L^{\infty}(0,T\,;L^2(\Omega))}^2+\|u-f\|_{L^2(0,T\,;H^s(\R^n))}^2\\
			\leq& C\left(\|u_0\|_{L^2(\Omega)}^2+\|f(0)\|_{L^2(\Omega)}^2+\|F\|^2_{L^2(0,T\,;H^{-s}(\Omega))}+\|\partial_tf\|^2_{L^2(0,T\,;H^{-s}(\Omega))} \right.\\
			&\quad \left. +\|\Div_s(\Theta_{\gamma}\nabla^sf)\|^2_{L^2(0,T\,;H^{-s}(\Omega))}\right).
		\end{split}
		\]

		Now, by the equation again, one knows that 
		\begin{align}
			\langle  \p_t (u-f), \varphi \rangle_{H^{-s}(\Omega)\times \widetilde{H}^s(\Omega)}+B_{\gamma}(t\,; u-f,\varphi)=\langle \widetilde{F},\varphi\rangle_{H^{-s}(\Omega)\times \widetilde{H}^s(\Omega)},
		\end{align} 
		for any $\varphi \in \wt H^s(\Omega)$ and a.e. $t\in (0,T)$. Consequently, 
		\begin{align}
			\left| \langle  \p_t (u-f), \varphi \rangle_{H^{-s}(\Omega)\times \widetilde{H}^s(\Omega)}  \right|\leq  \left| B_{\gamma}(t\,; u-f,\varphi)\right|+ \left|\langle \widetilde{F},\varphi\rangle_{H^{-s}(\Omega)\times \widetilde{H}^s(\Omega)}\right|,
		\end{align}
		Integrating from $0$ to $T$ and using $(L^2(0,T\,;\wt H^s(\Omega)))^{\ast}=L^2(0,T\,;H^{-s}(\Omega))$, one can conclude that $\p_t (u-f) \in L^2(0,T;\,  H^{-s}(\Omega))$. Since $\partial_tf\in L^2(0,T\,;H^{-s}(\Omega))$ we obtain $\p_t u \in L^2(0,T;\,  H^{-s}(\Omega))$ as desired. \\

		(ii): First we show that $\Div_s(\Theta_{\gamma}\nabla^sf)\in L^2(\Omega_T)$. More concretely, we prove that there holds
		\[
		\left|\int_0^T\langle \Div_s(\Theta_{\gamma}\nabla^sf),\varphi\rangle\,dt\right|\leq C \|f\|_{L^2(0,T\,;H^{2s}(\R^n))}\|\varphi\|_{L^2(\Omega_T)}
		\]
		for all $\varphi\in C_c^{\infty}(\Omega_T)$ and some $C>0$ independent of $\varphi$. This gives already the claim as then by density $\Div_s(\Theta_{\gamma}\nabla^sf)$ can be uniquely extended to an element in $L^2(\Omega_T)$ such that
		\begin{equation}
			\label{eq: L2 estimate for cond op}
			\|\Div_s(\Theta_{\gamma}\nabla^sf)\|_{L^2(\Omega_T)}\leq C\|f\|_{L^2(0,T\,;H^{2s}(\R^n))}.
		\end{equation}
		Using \cite[Remark 8.8]{RZ2022unboundedFracCald} in every time slice, we obtain
		\begin{equation}
			\label{eq: regularity estimate}
			\begin{split}
				&\left|\int_0^T\langle \Div_s(\Theta_{\gamma}\nabla^sf),\varphi\rangle\,dt\right|\\
				=&\left|\int_0^T\langle \Theta_{\gamma}\nabla^sf,\nabla^s\varphi\rangle\,dt\right|\\
				=&\left|\int_0^T\langle (-\Delta)^{s/2}(\gamma^{1/2}f),(-\Delta)^{s/2}(\gamma^{1/2}\varphi)\rangle+\langle q_{\gamma}(\gamma^{1/2}f),\gamma^{1/2}\varphi\rangle\,dt\right|.
			\end{split}
		\end{equation}
		Now note that by \cite[Corollary~A.7]{RZ2022unboundedFracCald} we have $\gamma^{1/2}\varphi\in H^s(\R^n)$.

		On the other hand, choosing $p_1=\frac{n}{2s},s_1=4s,p_2=2,r_2=\frac{2n}{n-2s}$ as in \cite[Lemma~A.6]{RZ2022unboundedFracCald}, using the Sobolev embedding and the monotonicity of Bessel potential spaces, we deduce that $m_{\gamma}f\in H^{2s}(\R^n)$ with
		\begin{equation}
			\begin{split}
				&\|m_{\gamma}f\|_{H^{2s}(\R^n)}\\
				\leq & C\LC\|m_{\gamma}\|_{L^{\infty}(\R^n)}\|f\|_{H^{2s}(\R^n)}+\|f\|_{L^{\frac{2n}{n-2s}}(\R^n)}\|m_{\gamma}\|^{1/2}_{H^{4s,\frac{n}{2s}}(\R^n)}\|m_{\gamma}\|_{L^{\infty}(\R^n)}^{1/2}\RC\\
				\leq  & C\|m_{\gamma}\|_{L^{\infty}(\R^n)}\LC 1+\|m_{\gamma}\|^{1/2}_{H^{4s,\frac{n}{2s}}(\R^n)}\RC\|f\|_{H^{2s}(\R^n)}.
			\end{split}
		\end{equation}
		This in turn shows $\gamma^{1/2}f\in H^{2s}(\R^n)$ for a.e. $t\in (0,T)$ with
		\begin{equation}
			\label{eq: higher order estimate} 
			\|\gamma^{1/2}f\|_{H^{2s}(\R^n)}\leq \LC 1+\|m_{\gamma}\|_{L^{\infty}(\R^n)}\RC \LC 1+\|m_{\gamma}\|^{1/2}_{H^{4s,\frac{n}{2s}}(\R^n)}\RC \|f\|_{H^{2s}(\R^n)}.
		\end{equation}
		Additonally, by the Gagliardo--Nirenberg inequality in Bessel potential spaces and the Sobolev embedding (cf.~\cite[eq.~(18)]{StabilityFracCond}), we have
		\[
		\|m_{\gamma}\|_{H^{2s,n/s}(\R^n)}\leq C\|m_{\gamma}\|^{1/2}_{H^{4s,\frac{n}{2s}}(\R^n)}\|m_{\gamma}\|^{1/2}_{L^{\infty}(\R^n)}
		\]
		and thus $(-\Delta)^sm_{\gamma}\in L^{n/s}(\R^n)$. Applying H\"older's inequality with $p_1=n/s,p_2=\frac{2n}{n-2s},p_3=2$, we can estimate
		\begin{equation}
			\label{eq: potential estimate}
			\begin{split}
				&\|q_{\gamma}(\gamma^{1/2}f)\gamma^{1/2}\varphi\|_{L^1(\R^n)}\\
				\leq &\|q_{\gamma}\|_{L^{n/s}(\R^n)}\|\gamma^{1/2}f\|_{L^{\frac{2n}{n-2s}}(\R^n)}\|\gamma^{1/2}\varphi\|_{L^2(\R^n)}\\
				\leq &C\|m_{\gamma}\|_{H^{4s,\frac{n}{2s}}(\R^n)}^{1/2}\|m_{\gamma}\|_{L^{\infty}(\R^n)}\|\gamma^{1/2}f\|_{H^{s}(\R^n)}\|\varphi\|_{L^2(\R^n)}\\
				\leq & C\LC 1+\|m_{\gamma}\|_{H^{4s,\frac{n}{2s}}(\R^n)}^{1/2}\RC \|m_{\gamma}\|^2_{L^{\infty}(\R^n)}\LC 1+\|m_{\gamma}\|^{1/2}_{H^{2s,n/2s}(\R^n)}\RC \\
				&\qquad \cdot \|f\|_{H^s(\R^n)}\|\varphi\|_{L^2(\R^n)}
			\end{split}
		\end{equation}
		where in the third inequality we again used \cite[Corollary~A.7]{RZ2022unboundedFracCald}. Now using $\gamma^{1/2}f\in H^{2s}(\R^n),\gamma^{1/2}\varphi\in H^s(\R^n)$ and the estimates \eqref{eq: higher order estimate}, \eqref{eq: potential estimate}, we obtain by H\"older's inequality from \eqref{eq: regularity estimate} the bound
		\begin{equation}
			\label{eq: L2 bound}
			\begin{split}
				\left|\int_0^T\langle \Div_s(\Theta_{\gamma}\nabla^sf),\varphi\rangle\,dt\right|=&\left|\int_0^T\langle (-\Delta)^{s}(\gamma^{1/2}f),\gamma^{1/2}\varphi\rangle+\langle q_{\gamma}(\gamma^{1/2}f),\gamma^{1/2}\varphi\rangle\,dt\right|\\
				\leq & C\LC 1+\|\gamma\|_{L^{\infty}(\R^n_T)}\RC \LC 1+\|m_{\gamma}\|_{L^{\infty}(0,T\,;H^{4s,\frac{n}{2s}}(\R^n))}\RC\\
				&\quad \cdot \|f\|_{L^2(0,T\,;H^{2s}(\R^n))}\|\varphi\|_{L^2(\Omega_T)}.
			\end{split}
		\end{equation}
		
		On the other hand by definition there holds $u_0-f(0)\in \widetilde{H}^s(\Omega)$.
		Hence, if we set as above $\widetilde{u}=u-f$, then we see that it solves
		\eqref{eq: weak solutions of homogeneous conductivity equation existence} with 
		$\widetilde{u}_0\vcentcolon = u_0-f(0)\in \widetilde{H}^s(\Omega)$ and
		\begin{equation}
			\label{eq: regularity inhomo}
			\widetilde{F}\vcentcolon = F-\partial_tf-\Div_s(\Theta_{\gamma}\nabla^sf)\in L^2(\Omega_T).
		\end{equation}
		Now we proceed similarly as in \cite[Chapter~7, Theorem~5]{EvansPDE}. For this purpose let us recall how the unique solution $\widetilde{u}$ in \cite[Chapter~XVIII, Section~3.1-3.2, Theorem~1 and 2]{DautrayLionsVol5} is constructed. Since $\widetilde{H}^s(\Omega)$ is a separable Hilbert space, the finite dimensional subspaces 
		\[
		\widetilde{H}^s_m\vcentcolon =\text{span}\{w_1,\ldots,w_m\}
		\]
		for $m\in\N$, where $(w_k)_{k\in\N}\subset \widetilde{H}^s(\Omega)$ is an orthonormal basis of $\widetilde{H}^s(\Omega)$, form a Galerkin approximation for $\widetilde{H}^s(\Omega)$. Observe by density of $\widetilde{H}^s(\Omega)$ in $L^2(\Omega)$ the family $(\widetilde{H}^s_m)_{m\in\N}$ are also a Galerkin approximation for $L^2(\Omega)$. By \cite[Chapter~XVIII, Section~3.1-3.2, Lemma~1]{DautrayLionsVol5} there are unique solutions $\widetilde{u}_m\in C([0,T]\,;\widetilde{H}^s_m)$ with $\partial_t\widetilde{u}_m\in L^2(0,T\,;\widetilde{H}^s_m)$ and
		\begin{equation}
			\label{eq: approximate PDE}
			\langle \partial_t\widetilde{u}_m,w_j\rangle +B_{\gamma}(t\,;\widetilde{u}_m,w_j)=\langle \widetilde{F},w_j\rangle
		\end{equation}
		for all $1\leq j\leq m$ and a.e. $t\in (0,T)$, where $\widetilde{u}_0^m\in \widetilde{H}^s_m$ are chosen in such a way that $\widetilde{u}_0^m\to \widetilde{u}_0$ in $L^2(\Omega)$.

		In fact, the solutions $\widetilde{u}_m$ can be written in the form 
		\[
		\widetilde{u}_m=\sum_{j=1}^mc_m^jw_j.
		\]
		Here $c_m=(c_m^1,\ldots,c_m^m)$ are absolutely continuous functions and solve 
		\[
		A_m\partial_tc_m+B_m(t)c_m= \widetilde{F}_m(t),\quad c_m(0)=\widetilde{u}_0^m,
		\]
		where $A_m\vcentcolon =(\langle w_i,w_j\rangle)_{1\leq i,j\leq m}$, $B_m(t)\vcentcolon =(B_{\gamma}(t\,;w_i,w_j))_{1\leq i,j\leq m}$ and $\widetilde{F}_m(t)=(\langle \widetilde{F}(t),w_j\rangle)_{1\leq j\leq m}$. We have $\|\widetilde{u}_0^m\|_{L^2(\Omega)}\leq c\|\widetilde{u}_0\|_{L^2(\Omega)}$ for some constant independent of $m$. Next observe that if $\widetilde{u}_0\in \widetilde{H}^s(\Omega)$, as in our case, then we can take
		\[
		\widetilde{u}_0^m=\sum_{j=1}^m\langle \widetilde{u}_0,w_j\rangle w_j\in \widetilde{H}^s_m
		\]
		and see that $\widetilde{u}_0^m\to \widetilde{u}_0$ in $H^s(\R^n)$ as $m\to \infty$. Moreover, this convergence implies
		\begin{equation}
			\label{eq: initial estimate}
			\|\widetilde{u}_0^m\|_{H^s(\R^n)}\leq c\|\widetilde{u}_0\|_{H^s(\R^n)}
		\end{equation}
		for some $c>0$ independent of $m$. Now fix $m\in\N$, multiply \eqref{eq: approximate PDE} by $c_m^j$ and sum $j$ over $\{1,\ldots,m\}$ to obtain
		\begin{equation}
			\label{eq: basic equation control time derivative}
			\langle \partial_t\widetilde{u}_m,\partial_t\widetilde{u}_m\rangle_{L^2(\Omega)} +B_{\gamma}(t\,;\widetilde{u}_m,\partial_t\widetilde{u}_m)=\langle \widetilde{F},\partial_t\widetilde{u}_m\rangle_{L^2(\Omega)},
		\end{equation}
		where we used \eqref{eq: regularity inhomo} and $\partial_t\widetilde{u}_m\in\widetilde{H}^s_m\subset L^2(\Omega)$.

		Observe that there holds
		\[
		\begin{split}
			\partial_tB_{\gamma}(t\,;\widetilde{u}_m,\widetilde{u}_m)=&2B_{\gamma}(t\,;\widetilde{u}_m,\partial_t\widetilde{u}_m)+\int_{\R^{2n}}\left[(\partial_t\gamma^{1/2}(x,t))\gamma^{1/2}(y,t)\right.\\
			&\left.+\gamma^{1/2}(x,t)\partial_t\gamma^{1/2}(y,t))\nabla^s\widetilde{u}_m\cdot\nabla^s\widetilde{u}_m \right] dxdy.
		\end{split}
		\]
		Hence, using the uniform ellipticity of $\gamma$, the Cauchy--Schwartz inequality and Young's inequality show
		\[
		\begin{split}
			&\|\partial_t\widetilde{u}_m\|^2_{L^2(\Omega)} +\partial_tB_{\gamma}(t\,;\widetilde{u}_m,\widetilde{u}_m)\\
			\leq &  C\LC \|\partial_t\gamma\|_{L^{\infty}(\R^n_T)}\|\gamma\|_{L^{\infty}(\R^n_T)}^{1/2}\|\widetilde{u}_m\|_{H^s(\R^n)}^2+\epsilon^{-1}\|\widetilde{F}\|^2_{L^2(\Omega)}\RC +\epsilon\|\partial_t\widetilde{u}_m\|^2_{L^2(\Omega)}
		\end{split}
		\]
		for some $C>0$ only depending on the ellipticity constant $\gamma_0$ and all $\epsilon>0$. Taking $\epsilon =1/2$, we can absorb the last term on the left hand side and after integrating over $(0,t)\subset (0,T)$, we obtain
		\[
		\begin{split}
			&\|\partial_t\widetilde{u}_m\|_{L^2(\Omega_t)}^2+B_{\gamma}(t;\widetilde{u}_m,\widetilde{u}_m)\\
			\leq & B_{\gamma}(0;\widetilde{u}_m,\widetilde{u}_m)+C\LC \|\partial_t\gamma\|_{L^{\infty}(\R^n_T)}\|\gamma\|_{L^{\infty}(\R^n_T)}^{1/2}\|\widetilde{u}_m\|_{L^2(0,T\,;H^s(\R^n))}^2+\|\widetilde{F}\|^2_{L^2(\Omega_T)}\RC.
		\end{split}
		\]
		Taking the supremum over $(0,T)$, using the uniform ellipticity of $\gamma$ and \eqref{eq: initial estimate} we get 
		\begin{equation}
			\label{eq: time derivative estimate m}
			\begin{split}
				&\|\partial_t\widetilde{u}_m\|_{L^2(\Omega_T)}^2+\|\widetilde{u}_m\|^2_{L^{\infty}(0,T\,;H^s(\R^n))}\\
				\leq & C\left( \|\gamma\|_{L^{\infty}(\R^n_T)}\|\widetilde{u}_0\|^2_{H^s(\R^n)}\right.\\
				& \quad \left. +\|\partial_t\gamma\|_{L^{\infty}(\R^n_T)}\|\gamma\|_{L^{\infty}(\R^n_T)}^{1/2}\|\widetilde{u}_m\|_{L^2(0,T\,;H^s(\R^n))}^2+\|\widetilde{F}\|^2_{L^2(\Omega_T)}\right).
			\end{split}
		\end{equation}
		Now the term $\|\widetilde{u}_m\|_{L^2(0,T\,;H^s(\R^n))}$ can be bounded from above using the energy estimate 
		\begin{equation}
			\label{eq: energy estimate for m}
			\|\widetilde{u}_m\|_{L^{\infty}(0,T\,;L^2(\Omega))}^2+\|\widetilde{u}_m\|^2_{L^{2}(0,T\,;H^s(\R^n))}\leq C\LC \|\widetilde{u}_0\|^2_{L^2(\Omega)}+\|\widetilde{F}\|^2_{L^2(\Omega_T)}\RC
		\end{equation}
		(cf.~\cite[Chapter~XVIII, Section~3.2, eq.~(3.40)]{DautrayLionsVol5}) for some $C>0$ only depending $\gamma_0$. This then gives
		\[
		\begin{split}
			\|\partial_t\widetilde{u}_m\|_{L^{2}(\Omega_T)}^2+\|\widetilde{u}_m\|^2_{L^{\infty}(0,T\,;H^s(\R^n))}\leq C\LC \|\widetilde{u}_0\|^2_{H^s(\R^n)}+\|\widetilde{F}\|^2_{L^2(\Omega_T)}\RC
		\end{split}
		\]
		for some $C>0$ only depending on $\gamma$.

		By \cite[Chapter~XVIII, Section~3.3, Lemma~3]{DautrayLionsVol5} we know that up to subsequences there holds
		\begin{enumerate}[(i)]
			\item $\widetilde{u}_m\weak \widetilde{u}$ in $L^2(0,T\,;H^s(\R^n))$
			\item and $\widetilde{u}_m\weakstar \widetilde{u}$ in $L^{\infty}(0,T\,;L^2(\Omega))$
		\end{enumerate}
		as $m\to \infty$.
		But \cite[Chapter~7, Problem~6]{EvansPDE} then implies (up to extracting possibly a further subsequence) that $\partial_t\widetilde{u}\in L^2(\Omega_T)$ and $\widetilde{u}\in L^{\infty}(0,T\,;H^s(\R^n))$ with
		\[
		\begin{split}
			&\|\partial_t\widetilde{u}\|_{L^{2}(\Omega_T)}^2+\|\widetilde{u}\|^2_{L^{\infty}(0,T\,;H^s(\R^n))} \\
			\leq &C \LC \|\widetilde{u}_0\|^2_{H^s(\R^n)}+\|\widetilde{F}\|^2_{L^2(\Omega_T)}\RC \\
			\leq & C \left( \|u_0\|_{H^s(\R^n)}^2+\|f(0)\|^2_{H^s(\R^n)}+\|F\|^2_{L^2(\Omega_T)} \right. \\
			& \qquad \left. +\|\partial_t f\|_{L^2(\Omega_T)}^2+\|f\|_{L^2(0,T\,;H^{2s}(\R^n))}^2\right),
		\end{split}
		\]
		where $C>0$ only depends on $\gamma$. Here we finally used the definition of $\widetilde{u}_0$, $\widetilde{F}$ and \eqref{eq: L2 estimate for cond op}. This establishes the estimate \eqref{eq: bound time derivate} and we can conclude the proof.
	\end{proof}
	
	Because of this well-posedness result we make the following definition:
	
	\begin{definition}
		\label{def: data space}
		Let $\Omega\subset\R^n$ be an open set, $0<T<\infty$, $0<s<\min(1,n/2)$ and $\gamma_0>0$. Then we define the data spaces $X_s(\Omega_T)$, $\widetilde{X}_s(\Omega_T)$ and the class of admissible conductivities $\Gamma_{s,\gamma_0}(\R^n_T)$ by
		\[
		\begin{split}
			X_s(\Omega_T)\vcentcolon = &\,\{(f,u_0)\in L^2(0,T\,;H^{2s}(\R^n))\times H^s(\R^n)\,;\\
			&\quad\partial_tf\in L^2(0,T\,;L^2(\R^n)),\,u_0-f(0)\in\widetilde{H}^s(\Omega)\,\},\\
			\widetilde{X}_s(\Omega_T)\vcentcolon =&\,\left\{f\in L^2(0,T\,;H^{2s}(\R^n))\,;\partial_tf\in L^2(0,T\,;L^2(\R^n)),\,f(0)=0\right\},\\
		\end{split}
		\] 
		and 
		\begin{align}\label{Gamma s gamma0}
			\begin{split}
				\Gamma_{s,\gamma_0}(\R^n_T)\vcentcolon =&\,\left\{\gamma\in  C^1_tC_x(\R^n_T)\,; \, \gamma \text{ satisfies \eqref{ellipticity}},\,  \partial_t\gamma\in L^{\infty}(\R^n_T),\right.\\
				&\left. \quad \text{ and }\ m_{\gamma}\in C([0,T]\,;H^{4s+\varepsilon,\frac{n}{2s}}(\R^n))\quad\text{for some}\quad\varepsilon>0 \right\}.
			\end{split}
		\end{align}
		Here, the space $C^k_tC^{\ell}_x(\R^n_T)$, $k,\ell\in\N_0$, consists of all functions which are $k$-times continuously differentiable in the time variable $t$ and $\ell$-times in the space variable $x$.
	\end{definition}
	
	With this notation at hand, the above theorem can be rewritten as the following:
	\begin{corollary}
		\label{cor: well-posedness}
		Let $\Omega\subset\R^n$ be an open set, $0<T<\infty$, $0<s<\min(1,n/2)$ and $\gamma_0>0$. Then for all $(f,u_0)\in X_s(\Omega_T)$ and $\gamma\in \Gamma_{s,\gamma_0}(\R^n_T)$ there is a unique solution $u_{f,u_0}\in L^{\infty}(0,T\,;H^s(\R^n))$ with $\partial_tu_{f,u_0}\in L^2(\Omega_T)$ satisfying
		\begin{equation}
			\label{eq: corollary bound time derivate}
			\begin{split}
				&\|\partial_t (u_{f,u_0}-f)\|^2_{L^2(\Omega_T)}+\|u_{f,u_0}-f\|^2_{L^{\infty}(0,T\,;H^s(\R^n))}\\
				\leq & C\LC \|u_0\|_{H^s(\R^n)}^2+\|f(0)\|^2_{H^s(\R^n)}+\|\partial_t f\|_{L^2(\Omega_T)}^2+\|f\|_{L^2(0,T\,;H^{2s}(\R^n))}^2\RC ,
			\end{split}
		\end{equation}
		for some constant $C>0$ independent of $f,u_0$ and $u_{f,u_0}$.
	\end{corollary}
	\begin{proof}
		This is an immediate consequence of Theorem~\ref{thm: well-posedness of forward problem} by taking $F=0$.
	\end{proof}

	With the well-posedness at hand, we can define the DN map \eqref{DN map in intro}, which was introduced in Section \ref{sec: introduction}, rigorously. Similarly as in the nonlocal elliptic case (see \cite{XavierNeumannBdry} or Appendix \ref{sec: Discussion of nonlocal normal derivatives and DN maps}) we define:
	
	\begin{definition}[The DN map]\label{def: the DN map}
		Let $\Omega\subset \R^n$ be an open set bounded in one direction, $0<T<\infty$, $0<s<\min(1,n/2)$, $\gamma_0>0$ and $\gamma\in \Gamma_{s,\gamma_0}(\R^n_T)$. Then we define the DN map $\Lambda_{\gamma}$ by
		\begin{equation}
			\label{eq: modified DN map}
			\begin{split}
				\langle \Lambda_{\gamma}f,g\rangle\vcentcolon =&\int_0^T B_{\gamma}(u_f,g)\,dt\\
				=&\frac{C_{n,s}}{2}\int_0^T\int_{\R^{2n}}\gamma^{1/2}(x,t)\gamma^{1/2}(y,t)\\
				&\qquad \qquad \quad \cdot \frac{(u_f(x,t)-u_f(y,t))(g(x,t)-g(y,t))}{|x-y|^{n+2s}}\,dxdydt
			\end{split}
		\end{equation}
		for all $f,g\in C_c^{\infty}((\Omega_e)_T)$, where $u_f$ is the unique solution of \eqref{main eq nonlocal diff}. 
	\end{definition}

	\section{Exterior determination} \label{sec: exterior det}
	The main goal of this section is to prove Theorem \ref{Theorem: Exterior determination}. We first establish an energy estimate which allows us to deduce that the Dirichlet energies of suitable special solutions concentrate in the exterior. 

	\begin{lemma}
		\label{lemma: elliptic/parabolic estimate} 
		Suppose that $W \subset \Omega_e$ is an open nonempty set with finite measure and $\text{dist}(W,\Omega) > 0$. Let $u_f$ be the unique solution to \eqref{eq: weak solutions of conductivity equation existence} with $f \in C_c^\infty(W_T)$, $F\equiv0$ and $u_0\equiv0$. Then
		\begin{equation}
			\label{eq: parabolic energy estimate}
			\begin{split}
				&\|u_f-f\|_{L^{\infty}(0,T\,;L^2(\Omega))}+\|u_f-f\|_{L^2(0,T\,;H^s(\R^n))}\leq C\|f\|_{L^2(W_T)},
			\end{split}
		\end{equation}
		where the constant $C>0$ does not depend on $f \in C_c^\infty(W_T)$.
	\end{lemma}
	\begin{proof}
		By applying the energy estimate \eqref{eq: energy estimate} in Theorem~\ref{thm: well-posedness of forward problem}, we obtain
		\[
		\begin{split}
			&\|u_f-f\|_{L^{\infty}(0,T\,;L^2(\Omega))}^2+\|u_f-f\|_{L^2(0,T\,;H^s(\R^n))}^2\\
			\leq & C\LC \|\partial_tf\|^2_{L^2(0,T\,;H^{-s}(\Omega))} +\|\Div_s(\Theta_{\gamma}\nabla^sf)\|^2_{L^2(0,T\,;H^{-s}(\Omega))}\RC .
		\end{split}
		\]
		Since, $f$ is compactly supported in $W_T\subset (\Omega_e)_T$ the first contribution in the above estimate is zero. By \cite[Proof of Lemma~3.1]{RZ2022LowReg} there holds 
		\[
		\|\Div_s(\Theta_{\gamma}\nabla^sf)(t)\|_{H^{-s}(\Omega)}\leq C\|f(t)\|_{L^2(W)}
		\]
		for a.e. $t\in (0,T)$ and some $C>0$ only depending on $n,s,W$ and $\|\gamma\|_{L^{\infty}(\R^n_T)}$. Hence, there holds
		\[
		\begin{split}
			&\|u_f-f\|_{L^{\infty}(0,T\,;L^2(\Omega))}^2+\|u_f-f\|_{L^2(0,T\,;H^s(\R^n))}^2\leq C\|f\|^2_{L^2(W_T)}
		\end{split}
		\]
		and we can conclude the proof.
	\end{proof}

	\begin{proof}[Proof of Theorem \ref{Theorem: Exterior determination}] 
		First, let $\gamma$ denote either of the two diffusion coefficients $\gamma_1$ or $\gamma_2$. Using the Sobolev embedding, we may assume that $\gamma\in C_b(W_T)$. By \cite[Lemma~5.5]{CRZ2022global}, for any $x_0 \in W$, there exists $(\phi_N)_{N \in \N} \subset C_c^\infty(W)$ such that $\|\phi_N\|_{H^s(\R^n)}=1$, $\|\phi_N\|_{L^2(\R^n)}\to 0$ as $N\to\infty$ and $\text{supp}(\phi_N)\to \{x_0\}$. Moreover, \cite[Proposition~1.5]{CRZ2022global} implies that
		$B_{\gamma(\cdot,t_0)}(\phi_N,\phi_N) \to \gamma(x_0,t_0)$ as $N\to\infty$ for any $t_0\in (0,T)$. Next let $\eta \in C_c^\infty((0,T))$ and define $\Phi_N \vcentcolon = \eta\phi_N \in C_c^\infty(W_T)$. It follows that
		\begin{equation}
			\int_0^T B_{\gamma}(\Phi_N,\Phi_N)\, dt = \int_0^T \eta^2(t)B_{\gamma}(\phi_N,\phi_N)\, dt.
		\end{equation}
		By the dominated convergence theorem we obtain
		\begin{equation}\label{eq: exterior reconstruction formula}
			\lim_{N\to \infty}\int_0^T B_{\gamma}(\Phi_N,\Phi_N)\, dt = \int_0^T \eta^2(t)\gamma(x_0,t)\, dt.
		\end{equation}
		
		Let us now consider the solutions $u_N$ to the equation 
		\begin{align}\label{eq:u_f modified DN map}
			\begin{cases}
				\partial_tu +\Div_s(\Theta_{\gamma}\nabla^s u)=0 & \text{ in }\Omega_T,\\
				u=\Phi_N & \text{ in }(\Omega_e)_T,\\
				u(x,0)=0 &\text{ in }\Omega,
			\end{cases}
		\end{align}
		for $N\in\N$. By the definition of the DN map \eqref{eq: modified DN map}, we have
		\begin{equation}
			\label{eq:splitting of energy}
			\begin{split}
				\left\langle\Lambda_{\gamma}\Phi_N,\Phi_N \right\rangle= &\int_{0}^TB_{\gamma}(u_N,\Phi_N)\,dt\\
				=&\int_{0}^TB_{\gamma}(u_N-\Phi_N,\Phi_N)\,dt+\int_{0}^TB_{\gamma}(\Phi_N,\Phi_N)\,dt
			\end{split}
		\end{equation}
		Next, note that Lemma \ref{lemma: elliptic/parabolic estimate} implies \begin{equation}
			\begin{split}
				\abs{\int_0^T B_{\gamma}(u_N-\Phi_N,\Phi_N)dt} &\leq C\int_0^T \norm{(u_N-\Phi_N)(\cdot,t)}_{H^s(\R^n)}\norm{\Phi_N(\cdot,t)}_{H^s(\R^n)}dt\\
				&\leq C\left(\int_0^T\norm{(u_N-\Phi_N)(\cdot,t)}_{H^s(\R^n)}^2dt\right)^{1/2} \\
				&\leq C\norm{\Phi_N}_{L^2(0,T;L^2(\R^n))}\\
				&=C\|\eta\|_{L^2((0,T))}\|\phi_N\|_{L^2(W)},
			\end{split}
		\end{equation}
		and hence there holds
		\begin{equation}
			\label{eq: limit remainder}
			\lim_{N \to \infty}\int_0^T B_{\gamma}(u_N-\Phi_N,\Phi_N)\,dt=0.
		\end{equation} 

		We obtain from \eqref{eq: exterior reconstruction formula},  \eqref{eq:splitting of energy} and \eqref{eq: limit remainder} that
		\begin{equation}
			\label{eq: weak squared reconstruction formula}
			\lim_{N\to \infty} \left\langle\Lambda_{\gamma}\Phi_N,\Phi_N \right\rangle = \int_0^T \eta^2(t)\gamma(x_0,t)\, dt.
		\end{equation}
		Hence, applying the identity \eqref{eq: weak squared reconstruction formula} to $\gamma= \gamma_1$ and $\gamma=\gamma_2$, and subtracting them, with \eqref{same DN in thm 1} at hand, we deduce
		\begin{equation}
			\label{eq: identity}
			\int_0^T (\gamma_1(x_0,t)-\gamma_2(x_0,t))\eta\, dt=0
		\end{equation}
		for all $\eta\in C_c^{\infty}((0,T))$ with $\eta\geq 0$. This implies $\gamma_1(x_0,t)\geq \gamma_2(x_0,t)$ a.e. Interchanging the role of $\gamma_1$ and $\gamma_2$, we also obtain the reversed inequality and deduce by continuity that $\gamma_1(x_0,t)=\gamma_2(x_0,t)$ for all $t\in (0,T)$.
		Since this construction can be done for any $x_0\in W$, we have $\gamma_1=\gamma_2$ in $W_T$.
	\end{proof}
	\begin{remark} 
		Note that we also obtain a Lipschitz stability estimate for the exterior determination problem with partial data as in the elliptic case \cite{CRZ2022global}. Moreover, one can easily observe that in contrast to the DN map $\Lambda_{\gamma}$ the new DN map $\mathcal{N}_{\gamma}$, which is defined in Section~\ref{sec: new DN maps}, satisfies $\lim_{N\to\infty}\langle \mathcal{N}_{\gamma}\Phi_N,\Phi_N\rangle =0$.
	\end{remark}

	\section{The spacetime Liouville reduction}
	\label{sec: DN and Liouville}
	
	In this section, we derive the spacetime Liouville reduction.

	\begin{lemma}[Auxiliary Lemma]
		\label{useful lemma}
		Let $\Omega\subset \R^n$ be an open set bounded in one direction, $0<T<\infty$, $0<s<\min(1,n/2)$, and $V\subset \Omega_e$ a nonempty open set. Assume that $\gamma \in L^\infty(\R^n_T)$ with background deviation $m_{\gamma}\in L^{\infty}(0,T\,;H^{2s,\frac{n}{2s}}(\R^n))$ satisfies $\gamma \geq \gamma_0>0$ for some positive constant $\gamma_0$. Then the following assertions hold:
		\begin{enumerate}[(i)]
			\item\label{assertion 1} For any $\psi\in L^2(0,T\,;\widetilde{H}^s(V))$, we have $\gamma^{1/2}\psi, \gamma^{-1/2}\psi \in L^2(0,T\,;\widetilde{H}^s(V))$ and there holds
			\begin{equation}
				\label{estimate 1}
				\begin{split}
					\|\gamma^{1/2}\psi\|_{L^2(0,T\,;H^s(\R^n))}&\lesssim \LC 1+\|m_{\gamma}\|_{L^{\infty}(\R^n_T)}+\|m_{\gamma}\|_{L^{\infty}(0,T\,;H^{2s,\frac{n}{2s}}(\R^n))}\RC \\
					&\qquad\cdot\|\psi\|_{L^2(0,T\,;H^s(\R^n))}
				\end{split}
			\end{equation}
			and
			\begin{equation}
				\label{estimate 2}
				\begin{split}
					\|\gamma^{-1/2}\psi\|_{L^2(0,T\,;H^s(\R^n))}&\lesssim \left( 1+\|m_{\gamma}\|_{L^{\infty}(\R^n_T)}+\|m_{\gamma}\|_{L^{\infty}(0,T\,;H^{2s,\frac{n}{2s}}(\R^n))}\right.\\
					&\quad \left. \quad +\|m_{\gamma}\|_{L^{\infty}(0,T\,;H^{2s,\frac{n}{2s}}(\R^n))}^{2s}\right)\|\psi\|_{L^2(0,T\,;H^s(\R^n))}.
				\end{split}
			\end{equation}
			\item\label{assertion 2} Let $u,\varphi\in L^2(0,T\,;H^s(\R^n))$. Then there holds
			\begin{equation}
				\label{eq: useful identity}
				\begin{split}
					\int_{t_1}^{t_2}\langle \Theta_\gamma \nabla ^s u, \nabla ^s \varphi \rangle_{L^2(\R^{2n})}\,dt
					&= \int_{t_1}^{t_2}\langle  (-\Delta)^{s/2}(\gamma^{1/2}u) , (-\Delta)^{s/2}(\gamma^{1/2}\varphi) \rangle_{L^2(\R^n)}\,dt\\
					&\quad +\int_{t_1}^{t_2} \langle  q_\gamma \gamma^{1/2}u,\gamma^{1/2}\varphi \rangle_{L^2(\R^{n})} \,dt
				\end{split}
			\end{equation}
			for all $0\leq t_1<t_2\leq T$.
		\end{enumerate}
	\end{lemma}
	
	\begin{proof}
		(i): First we show that $\gamma^{1/2}\psi\in L^2(0,T\,;\widetilde{H}^s(V))$ for any $\psi\in L^2(0,T\,;\widetilde{H}^s(V))$. Decomposing $\gamma^{1/2}\psi$ as $m_{\gamma}\psi+\psi$, we deduce from \cite[Corollary~A.7]{RZ2022unboundedFracCald} that there holds
		\begin{equation}
			\label{eq: estimate for product}
			\begin{split}
				&\quad \|\gamma^{1/2}\psi\|^2_{L^2(0,T\,;H^s(\R^n))}\\
				&\leq C\LC \|m_{\gamma}\psi\|^2_{L^2(0,T\,;H^s(\R^n))}+\|\psi\|^2_{L^2(0,T\,;H^s(\R^n))}\RC\\
				&\leq  C\int_0^T\LC \|m_{\gamma}\|^2_{L^{\infty}(\R^n)}+\|m_{\gamma}\|_{H^{2s,n/2s}(\R^n)}\|m_{\gamma_i}\|_{L^{\infty}(\R^n)}\RC\|\psi\|^2_{H^s(\R^n)}\,dt\\
				&\quad +C\|\psi\|^2_{L^2(0,T\,;H^s(\R^n))}\\
				&\leq C \LC 1+\|m_{\gamma}\|_{L^{\infty}(\R^n_T)}^2+\|m_{\gamma}\|_{L^{\infty}(0,T\,;H^{2s,\frac{n}{2s}}(\R^n))}\RC \|\psi\|^2_{L^2(0,T\,;H^s(\R^n))}.
			\end{split}
		\end{equation}
		Hence, we have $\gamma^{1/2}\psi\in L^2(0,T\,;H^s(\R^n))$.

		Next, recall that if $T>0$ and $X$ is a Banach space with dense subset $X_0$, then $C_c^{\infty}((0,T))\otimes X_0$ is dense in $L^2(0,T\,;X)$. Let $(\rho_{\epsilon})_{\epsilon>0}\subset C_c^{\infty}(\R^n)$ be a standard mollifier and choose a sequence $(\psi_k)_{k\in\N}\subset C_c^{\infty}((0,T))\otimes C_c^{\infty}(V)$ such that $\psi_k\to \psi$ in $L^2(0,T\,;\widetilde{H}^s(V))$. The sequence $(\gamma^{1/2}\ast \rho_{\epsilon_k})\psi_k$, $k\in\N$, belongs to $L^2(0,T\,;\widetilde{H}^s(V))$, where $\epsilon_k\to 0$ as $k\to\infty$. Hence, if we can show that $(\gamma_i^{1/2}\ast \rho_{\epsilon_k})\psi_k\to \gamma_i^{1/2}\psi$ in $L^2(0,T\,;H^s(\R^n))$ as $k\to \infty$, then it follows that $\gamma^{1/2}\psi\in L^2(0,T\,;\widetilde{H}^s(V))$. We can estimate
		\begin{equation}\label{some estimeate in 5.1}
			\begin{split}
				&\|\gamma^{1/2}\psi-(\gamma^{1/2}\ast\rho_{\epsilon_k})\psi_k\|_{L^2(0,T\,;H^s(\R^n))}\\
				\leq & \|m_{\gamma}\psi-m_{\gamma}^k\psi_k\|_{L^2(0,T\,;H^s(\R^n))}+\|\psi-\psi_k\|_{L^2(0,T\,;H^s(\R^n))}\\
				\leq & \|(m_{\gamma}-m_{\gamma}^k)\psi\|_{L^2(0,T\,;H^s(\R^n))}+\|m_{\gamma}^k(\psi-\psi_k)\|_{L^2(0,T\,;H^s(\R^n))}\\
				& +\|\psi-\psi_k\|_{L^2(0,T\,;H^s(\R^n))},
			\end{split}
		\end{equation}
		where we have set $m_{\gamma}^k= m_{\gamma_i}\ast \rho_{\epsilon_k}$.

		Now, for the second term in the right hand side of \eqref{some estimeate in 5.1}, we can apply the estimate \eqref{eq: estimate for product}, but all the terms involving $m_{\gamma}^k$ are uniformly bounded for $k\in \N$ by using the Young's inequality and the fact that Bessel potentials commute with convolution. Hence, the second and third term go to zero as $k\to\infty$. For the first term in the right hand side of \eqref{some estimeate in 5.1}, we observe that by \cite[Corollary~A.7]{RZ2022unboundedFracCald}, there holds $(m_{\gamma}-m_{\gamma}^k)\psi\to 0$ in $H^s(\R^n)$ as $k\to\infty$ for a.e. $t\in (0,T)$ and
		\begin{align}
			&\|(m_{\gamma}-m_{\gamma}^k)\psi\|_{H^s(\R^n)}\\
			\leq & C\LC\|m_{\gamma}\|_{L^{\infty}(\R^n)}+\|m_{\gamma}\|_{H^{2s,n/2s}(\R^n)}^{1/2}\|m_{\gamma}\|_{L^{\infty}(\R^n)}^{1/2}\RC \|\psi\|_{H^s(\R^n)},
		\end{align}
		for a.e. $t\in (0,T)$. 
		
		With the above estimate at hand, let us use Young's inequality again and the Bessel potentials commute with convolution. Since, $m_{\gamma}\in L^{\infty}(0,T;H^{2s,\frac{n}{2s}}(\R^n))$ the term in brackets is uniformly bounded in $t$ and thus Lebesgue's dominated convergence theorem implies that $(m_{\gamma}-m_{\gamma}^k)\psi\to 0$ in $L^2(0,T\,;H^s(\R^n))$ as $k\to\infty$. Therefore, the assertion follows. 
		
		Similarly, one can prove $\gamma^{-1/2}\psi\in L^2(0,T\,;\widetilde{H}^s(V))$ for any $\psi\in L^2(0,T\,;\widetilde{H}^s(V))$. Indeed, it essentially follows from the decomposition $\gamma^{-1/2}=1-\frac{m_{\gamma}}{m_{\gamma}+1}$ and the fact that the second term has exactly the same regularity properties as $m_{\gamma}$. More concretely, from \cite[Proof of Theorem~8.6]{RZ2022unboundedFracCald} and \cite[p.~156]{AdamsComposition} it follows that
		\[
		\begin{split}
			&\left\|\frac{m_{\gamma}}{m_{\gamma}+1}\right\|_{L^{\infty}(0,T\,;H^{2s,\frac{n}{2s}}(\R^n))}
			\\
			\leq & C\left(\|m_{\gamma}\|_{L^{\infty}(0,T\,;H^{2s,\frac{n}{2s}}(\R^n))}+\|m_{\gamma}\|_{L^{\infty}(0,T\,;H^{2s,\frac{n}{2s}}(\R^n))}^{2s}\right),
		\end{split}
		\]
		and hence we can repeat the above argument by using the smooth approximation the function $\frac{m_{\gamma}^k}{m_{\gamma}^k+1}$ this time. Thus, we conclude that $\gamma^{-1/2}\psi\in L^2(0,T\,;\widetilde{H}^s(V))$ for all $\psi\in L^2(0,T\,;\widetilde{H}^s(V))$.\\
		
		\noindent (ii): Note that due to our regularity assumptions we can apply \cite[Lemma 4.1]{CRZ2022global} or \cite[Remark 8.8]{RZ2022unboundedFracCald} in every time slice, to obtain
		\begin{align}\label{DN relations from Liouville}
			\begin{split}
				\langle \Theta_\gamma \nabla ^s u, \nabla ^s \varphi \rangle_{L^2(\R^{2n})}     
				=&\,\langle  (-\Delta)^{s/2}(\gamma^{1/2}u) , (-\Delta)^{s/2}(\gamma^{1/2}\varphi) \rangle_{L^2(\R^n)}\\
				&-\langle (-\Delta)^{s/2}m_\gamma, (-\Delta)^{s/2}(\gamma^{1/2}u\varphi) \rangle_{L^2(\R^n)}   \\
				=&\,\langle  (-\Delta)^{s/2}(\gamma^{1/2}u) , (-\Delta)^{s/2}(\gamma^{1/2}\varphi) \rangle_{L^2(\R^n)}\\
				&+ \langle  q_\gamma \gamma^{1/2}u,\gamma^{1/2}\varphi \rangle_{L^2(\R^{n})},
			\end{split}
		\end{align}
		for a.e. $t\in (0,T)$ and all $u,\varphi\in L^2(0,T\,;H^s(\R^n))$, where $m_\gamma$ and $q_\gamma$ are the functions defined by  \eqref{m_gamma} and \eqref{eq: space-time potential}, respectively. Finally, note that by the properties of the fractional Laplacian, the fact that $(u,v)\,\mapsto q_{\gamma}uv$ is bilinear and bounded as a map from $L^2(0,T\,;H^s(\R^n))\times L^2(0,T\,;H^s(\R^n))$ to $L^1(0,T\,; L^1(\R^n))$ (cf.~\cite[Corollary~A.11]{RZ2022unboundedFracCald}) and the assertion \ref{assertion 1} all terms appearing in the above identity are in $L^1((0,T))$.
	\end{proof}
	
	Now, we are ready to introduce the Liouville reduction.
	
	\begin{theorem}[Fractional spacetime Liouville reduction]
		\label{thm:fractionalLiouvilleReduction}
		Let $\Omega\subset \R^n$ be an open set bounded in one direction, $0<T<\infty$, $0<s<\min(1,n/2)$, $\gamma_0>0$ and $\gamma\in \Gamma_{s,\gamma_0}(\R^n_T)$.
		\begin{enumerate}[(i)]
			\item\label{item 1 Liouville reduction} If $F\in L^2(\Omega_T)$, $(f,u_0)\in X_s(\Omega_T)$ and $u$ is the unique solution to \eqref{eq: weak solutions of conductivity equation existence}, then $v\vcentcolon = \gamma^{1/2}u\in H^1(0,T\,;L^2(\Omega))\cap L^2(0,T\,;H^s(\R^n))$ solves 
			\begin{align}
				\label{eq: weak solutions of Schrodinger equation existence}
				\begin{cases}
					\partial_t\LC \gamma^{-1}v\RC +\LC (-\Delta)^s+Q_{\gamma}\RC v= G  & \text{ in }\Omega_T,\\
					v= g  & \text{ in } (\Omega_e)_T,\\
					v(0) =v_0 & \text{ in } \Omega,
				\end{cases}
			\end{align}
			with $G=\gamma^{-1/2}F\in L^2(\Omega_T)$, $(g,v_0)=(\gamma^{1/2}f,\gamma^{1/2}u_0)\in X_s(\Omega_T)$ and
			\begin{equation}\label{eq: space-time potential}
				Q_{\gamma}=q_{\gamma}-\frac{\partial_t\gamma}{2\gamma^{2}}\quad\text{with}\quad q_{\gamma}=-\frac{(-\Delta)^sm_\gamma}{\gamma^{1/2}}.
			\end{equation}
			\item\label{item 2 Liouville reduction} For all $G\in L^2(\Omega_T)$, $(g,v_0)\in X_s(\Omega_T)$, there is a unique solution $v\in H^1(0,T\,;L^2(\Omega))\cap L^2(0,T\,;H^s(\R^n))$ to \eqref{eq: weak solutions of Schrodinger equation existence}. Moreover, it is given by $v=\gamma^{-1/2}u$, where $u$ is the solution to \eqref{eq: weak solutions of Schrodinger equation existence} with $F=\gamma^{1/2}G$, $f=\gamma^{-1/2}g$ and $u_0=\gamma^{-1/2}v_0$.
		\end{enumerate}
	\end{theorem}
	
	\begin{proof}
		(i): Note that $G\in L^2(\Omega_T)$ and by Lemma~\ref{useful lemma}, we have $g\in L^2(0,T\,;H^s(\R^n))$, $v_0\in H^s(\R^n)$ and $v-g\in L^2(0,T\,;\widetilde{H}^s(\Omega))$. By the assumptions on $f,\gamma,u$, we have $\partial_t(\gamma^{1/2}f)\in L^2(\R^n_T)$ and $\partial_tv\in L^2(\Omega_T)$. 
		This implies that 
		$$
		v_0-g(0)=\gamma^{1/2}(0)(u_0-f(0))\in \widetilde{H}^s(\Omega).
		$$
		Therefore, we have shown that $(g,v_0)\in X_s(\Omega_T)$.

		Arguing as above we have $v$ has the same regularity properties as $u$. Thus, it remains to prove that $v$ solves \eqref{eq: weak solutions of Schrodinger equation existence}. By definition $u$ satisfies
		\begin{equation}
			\label{eq: u solves cond eq}
			-\int_{\Omega_T}u\partial_t\varphi\,dxdt+\int_0^T B_{\gamma}(t\,;u,\varphi)\,dt=\int_{\Omega_T} F\varphi\,dxdt+\int_{\Omega}u_0(x)\varphi(x,0)\,dx
		\end{equation}
		for all $\varphi\in C_c^{\infty}(\Omega\times [0,T))$.  Observe, as in the case $s=1$ (cf.~\cite[Lemma~4.12]{LadyzhenskajaParabolic}), the space $C_c^{\infty}(\Omega\times [0,T))$ is dense in 
		\[
		W_s(\Omega_T)\vcentcolon = \{\,\varphi\in L^2(0,T\,;\widetilde{H}^s(\Omega))\,;\,\partial_t\varphi \in L^2(\Omega_T)\,\text{and}\,\varphi(T)=0\,\},
		\]
		which is endowed with the natural norm 
		\[
		\|u\|_{W_s(\Omega_T)}^2\vcentcolon = \|\partial_tu\|_{L^2(\Omega_T)}^2+\|u\|_{L^2(0,T\,;H^s(\R^n))}^2.
		\]
		In fact, this can easily seen by using a cutoff function in time as in \cite[Exercise~8.8]{Brezis}, and using the density of $C_c^{\infty}((0,T))\otimes C_c^{\infty}(\Omega)$ in $L^2(0,T\,;\widetilde{H}^s(\Omega))$. 
		
		Now as we proved above, the space $W_s(\Omega_T)$ is invariant under multiplication with either $\gamma^{1/2}$ or $\gamma^{-1/2}$. Moreover, we have
		\begin{equation}
			\label{eq: time derivative}
			\begin{split}
				u\partial_t\varphi=&\frac{1}{\gamma^{1/2}}(\gamma^{1/2}u)\partial_t\frac{\gamma^{1/2}\varphi}{\gamma^{1/2}}\\
				=&\frac{1}{\gamma}(\gamma^{1/2}u)\partial_t(\gamma^{1/2}\varphi)-\frac{1}{2\gamma^{2}}(\gamma^{1/2}u)(\gamma^{1/2}\varphi)\partial_t\gamma,
			\end{split}
		\end{equation}
		for all $\varphi\in W_s(\Omega_T)$. Hence, using the (space) Liouville reduction (see \cite[Remark~8.8]{RZ2022unboundedFracCald}) in every time slice for all $\varphi\in W_s(\Omega_T)$, the identity \eqref{eq: u solves cond eq} implies
		\begin{equation}
			\label{eq: slicewise Liouville}
			\begin{split}
				&-\int_{\Omega_T}u\partial_t\varphi\,dxdt+\int_0^T \langle (-\Delta)^{s/2}(\gamma^{1/2}u),(-\Delta)^{s/2}(\gamma^{1/2}\varphi)\rangle\,dt \\
				&+\int_{\Omega_T}q_{\gamma}(\gamma^{1/2}u)(\gamma^{1/2}\varphi)\,dxdt\\
				=&\int_{\Omega_T} F\varphi\,dxdt +\int_{\Omega}\frac{\gamma^{1/2}(x,0)u_0(x)\gamma^{1/2}(x,0)\varphi(x,0)}{\gamma(x,0)}\,dx.
			\end{split}
		\end{equation}
		Inserting \eqref{eq: time derivative}, this shows
		\begin{equation}
			\label{eq: inserting time derivative}
			\begin{split}
				&-\int_{\Omega_T}\frac{(\gamma^{1/2}u)\partial_t(\gamma^{1/2}\varphi)}{\gamma}\,dxdt+\int_0^T \langle (-\Delta)^{s/2}(\gamma^{1/2}u),(-\Delta)^{s/2}(\gamma^{1/2}\varphi)\rangle\,dt\\
				&+\int_{\Omega_T}Q_{\gamma}(\gamma^{1/2}u)(\gamma^{1/2}\varphi)\,dxdt
				\\=&\int_{\Omega_T} (\gamma^{-1/2}F)(\gamma^{1/2}\varphi)\,dxdt\\
				& +\int_{\Omega}\frac{\gamma^{1/2}(x,0)u_0(x)\gamma^{1/2}(x,0)\varphi(x,0)}{\gamma(x,0)}\,dx
			\end{split}
		\end{equation}
		Hence, choosing $\varphi=\gamma^{-1/2}\psi$ with $\psi\in W_s(\Omega_T)$, we see that $v$ satisfies
		\begin{equation}
			\label{eq: replacing test function}
			\begin{split}
				&-\int_{\Omega_T}\gamma^{-1}v\partial_t\psi\,dxdt+\int_0^T \langle (-\Delta)^{s/2}v,(-\Delta)^{s/2}\psi\rangle\,dt+\int_{\Omega_T}Q_{\gamma}v\psi\,dxdt\\
				=&\int_{\Omega_T} G\psi\,dxdt +\int_{\Omega}\frac{v_0(x)\psi(x,0)}{\gamma(x,0)}\,dx, \\
			\end{split}
		\end{equation}
		for all $\psi\in W_s(\Omega_T)$. Therefore $v$ is a solution of \eqref{eq: weak solutions of Schrodinger equation existence} as claimed.\\
		
		\noindent (ii): Existence and uniqueness of solutions in $H^1(0,T\,;L^2(\Omega))\cap L^2(0,T\,;H^s(\R^n))$ easily follows from (i) by choosing the data $F,f,u_0$ appropriately and observing that $\gamma^{-1/2}$ has precisely the same regularity properties as $\gamma^{1/2}$. In fact, assume that $v\in H^1(0,T\,;L^2(\Omega))\cap L^2(0,T\,;H^s(\R^n))$ solves \eqref{eq: weak solutions of Schrodinger equation existence}. Then as in (i), we deduce $u\vcentcolon = \gamma^{-1/2}v\in H^1(0,T\,;L^2(\Omega))\cap L^2(0,T\,;H^s(\R^n))$, $F\vcentcolon = \gamma^{1/2}G\in L^2(\Omega)$ and $(f,u_0)\vcentcolon = (\gamma^{1/2}g,\gamma^{1/2}v_0)\in X_s(\Omega_T)$.

		By the definition, $v$ solves \eqref{eq: replacing test function} for any $\psi\in W_s(\Omega_T)$. Replacing $\psi$ by $\gamma^{1/2}\varphi$ and inserting these definitions of $F,u_0$ and $f$, we get \eqref{eq: inserting time derivative}. Plugging the identity \eqref{eq: time derivative} and using the slicewise Liouville reduction, we see that $u$ solves
		\begin{equation}
			\label{eq: exist schroed eq}
			\begin{cases}
				\partial_t u +\Div_s(\Theta_{\gamma}\nabla^s u)= F  & \text{ in }\Omega_T,\\
				u= f  & \text{ in } (\Omega_e)_T,\\
				u(0) =u_0 & \text{ in } \Omega.
			\end{cases}
		\end{equation}
		Now, Theorem~\ref{thm: well-posedness of forward problem} gives the existence of such a solution $u$ and by (i) the function $v$ solves \eqref{eq: weak solutions of Schrodinger equation existence}. On the othere hand if $v_1,v_2$ are two solutions of \eqref{eq: weak solutions of Schrodinger equation existence}, then arguing as above we see that $u_i\vcentcolon = \gamma^{-1/2}v_i$ for $i=1,2$ solve \eqref{eq: exist schroed eq}. Since solutions to the nonlocal diffustion equation \eqref{eq: exist schroed eq} are unique, we deduce that $u_1=u_2$ and thus $v_1=v_2$.
	\end{proof}

	Next we want to prove the well-posedness for the diffusion equation derived by the Liouville reduction and its adjoint equation under the milder assumption $G\in L^2(0,T\,;H^{-s}(\Omega))$ but $g\in C_c^{\infty}((\Omega_e)_T)$. Moreover, we will see that $u$ is the solution to \eqref{eq: weak solutions of conductivity equation existence} if and only if $v$ is the unique solution to \eqref{eq: weak solutions of Schrodinger equation existence} of the form $v=\gamma^{-1/2}u$.
	
	
	\begin{definition}
		\label{def: time reversal}
		If $u\in L^1_{loc}(V_T)$ for some open set $V\subset\R^n$, then we set
		\begin{equation}
			\label{eq: time reversal}
			u^{\ast}(x,t)\vcentcolon = u(x,T-t)
		\end{equation}
		for all $(x,t)\in V_T$.
	\end{definition}
	
	\begin{proposition}[Well-posedness]\label{prop: solution schroedinger eqs}
		Let $\Omega\subset \R^n$ be an open set bounded in one direction, $0<T<\infty$, $0<s<\min(1,n/2)$, $\gamma_0>0$ and $\gamma\in  \Gamma_{s,\gamma_0}(\R^n_T)$. If $g\in C_c^{\infty}((\Omega_e)_T)$ and $G\in L^2(0,T; \, H^{-s}(\Omega))$, then there exist unique solutions $v$, $v^\ast \in L^2(0,T\,;H^s(\R^n))$ with $\partial_tv$ and  $\p_t v^\ast \in L^2(0,T, ;  H^{-s}(\Omega))$
		of
		\begin{align}\label{Liouv reduced Schrodinger equation}
			\begin{cases}
				\partial_t(\gamma^{-1}v) +\LC (-\Delta)^s+Q_\gamma\RC v= G  & \text{ in }\Omega_T,\\
				v= g  & \text{ in } (\Omega_e)_T,\\
				v(0) =0 & \text{ in } \Omega,
			\end{cases}
		\end{align}
		and 
		\begin{align}\label{Liouv reduced Schrodinger equation adjoint}
			\begin{cases}
				-\gamma^{-1}\partial_tv^\ast  +\LC (-\Delta)^s+Q_\gamma\RC v^\ast = G  & \text{ in }\Omega_T,\\
				v^\ast = g  & \text{ in } (\Omega_e)_T,\\
				v^\ast(T) =0 & \text{ in } \Omega,
			\end{cases}
		\end{align}
		respectively. Here $Q_\gamma$ is the function \eqref{eq: space-time potential} given by the Liouville reduction.
	\end{proposition}
	
	\begin{proof}

		Let us prove the uniqueness of solutions to \eqref{Liouv reduced Schrodinger equation}. Suppose that $v_1,v_2$ are solutions of \eqref{Liouv reduced Schrodinger equation}, and consider $\wt v:=v_1 -v_2$, then $\wt v$ is the solution to 
		\begin{align}\label{wt v}
			\begin{cases}
				\p_t (\gamma^{-1} \wt v) + \LC (-\Delta)^s+ Q_\gamma \RC \wt v= 0 & \text{ in }\Omega_T,\\
				\wt v= 0 & \text{ in } (\Omega_e)_T,\\
				\wt v(0) =0 & \text{ in } \Omega.
			\end{cases}
		\end{align}
		Multiplying \eqref{wt v} by $\wt v$, then it is not hard to see 
		\begin{align}\label{uni 1}
			\int_{\Omega} \p_t (\gamma^{-1}\wt v) \wt v \, dx + \int_{\R^n}|(-\Delta)^{s/2}\wt v|^2\, dx +\int_{\Omega}   Q_\gamma |\wt v|^2 \, dx =0,  
		\end{align}
		where the first integral has to be understood as the duality pairing between $\widetilde{H}^s(\Omega)$ and $H^{-s}(\Omega)$.
		Meanwhile, notice that the first term of the above identity can be expressed as 
		\begin{align}\label{uni 2}
			\int_{\Omega} \p_t (\gamma^{-1}\wt v) \wt v \, dx= \frac{\p_t}{2}\int_{\Omega} \gamma^{-1}|\wt v|^2\, dx+\int_{\Omega}|\wt v|^2 \gamma^{-1/2}\p_t(\gamma^{-1/2})\, dx.
		\end{align}
		We next plug \eqref{uni 2} into \eqref{uni 1}, which give rises to 
		\begin{align}\label{uni 3}
			\begin{split}
				& \frac{\p_t}{2}\int_{\Omega} \gamma^{-1}|\wt v|^2\, dx + \int_{\R^n}|(-\Delta)^{s/2}\wt v|^2\, dx  \\
				=& -\int_{\Omega} Q_\gamma |\wt v|^2 \, dx -\int_{\Omega} \gamma^{-1/2}\p_t(\gamma^{-1/2})|\wt v|^2\, dx
				\\
				\leq &C \int_{\Omega}\gamma^{-1}|\wt v|^2 \, dx,
			\end{split}
		\end{align}
		for a constant $C>0$ independent of $\wt v$, where we used that $\gamma\in C^1_tC_x(\R^n_T)$ is uniformly elliptic with $\partial_t \gamma\in L^{\infty}(\R^n_T)$. 
		
		Thus, we obtain
		\begin{align}
			\p_t\norm{\gamma^{-1/2}\wt v}_{L^2(\Omega)}^2 \leq & C\left(\p_t\int_{\Omega} \gamma^{-1}|\wt v|^2\, dx + \int_{\R^n}|(-\Delta)^{s/2}\wt v|^2\, dx\right) \\
			\leq & C \norm{\gamma^{-1/2}\wt v}_{L^2(\Omega)}^2, 
		\end{align}
		and the Gronwall's inequality implies that 
		\[
		\norm{\gamma^{-1/2}(\cdot, t)\wt v(\cdot, t)}^2_{L^2(\Omega)} \leq e^{Ct}\norm{\gamma^{-1/2}(\cdot, 0)\wt v(\cdot,0)}^2_{L^2(\Omega)}=0, \text{ for }t\in (0,T),
		\]
		where we used the initial condition is $0$. This shows $\wt v=0$ in $\Omega_T$ as desired.
		
		When $\gamma\in C^1_tC_x(\R^n_T)$ is uniformly elliptic with $\partial_t\gamma\in L^{\infty}(\R^n_T)$, the proof of well-posedness of either \eqref{Liouv reduced Schrodinger equation} or \eqref{Liouv reduced Schrodinger equation adjoint} are similar. More precisely, one can use the relation
		\[
		\p_t (\gamma^{-1}v)=\gamma^{-1}\p_t v + \p_t(\gamma^{-1})v,
		\]
		then we are able to rewrite the equation \eqref{Liouv reduced Schrodinger equation} as 
		\begin{align}\label{Liouv reduced Schrodinger equation 1}
			\begin{cases}
				\gamma^{-1}\partial_t v +\LC (-\Delta)^s+\wt Q_\gamma\RC v= G  & \text{ in }\Omega_T,\\
				v= g  & \text{ in } (\Omega_e)_T,\\
				v(0) =0 & \text{ in } \Omega,
			\end{cases}
		\end{align}
		where $\wt Q_\gamma:=Q_\gamma+\p_t (\gamma^{-1})$ in $\Omega_T$. Now, it is not hard to see the well-posedness of \eqref{Liouv reduced Schrodinger equation 1} and \eqref{Liouv reduced Schrodinger equation adjoint} are the same by reversing the time variable $t\to T-t$ as in Definition \ref{def: time reversal}. 
		
		Now, by slight modification of the proof of Theorem~\ref{thm:fractionalLiouvilleReduction} one knows that $v$ is the unique solution to \eqref{Liouv reduced Schrodinger equation} if and only if $u$ is the solution to \eqref{eq: weak solutions of conductivity equation existence} with $G=\gamma^{-1/2}F\in L^2(0,T\,;H^{-s}(\Omega))$ and $g=\gamma^{1/2}f\in L^2(0,T\,;H^s(\R^n))$ with $\partial_tg\in L^2(\R^n_T)$. Hence, applying Theorem \ref{thm: well-posedness of forward problem} for the solution $u$ of \eqref{eq: weak solutions of conductivity equation existence}, one has $u\in L^2(0,T; H^s(\R^n))$ with $\p_t u \in L^2(0,T;H^{-s}(\Omega))$, so does $v$. This proves the assertion.
	\end{proof}

	\begin{remark}
		Note that combining similar arguments as in the proofs of Proposition \ref{prop: solution schroedinger eqs} and Theorem \ref{thm: well-posedness of forward problem}, one may derive the well-posedness of the initial-exterior value problem of 
		\begin{align}
			\begin{cases}
				\partial_t(\gamma^{-1}v) +\LC (-\Delta)^s+Q\RC v= G  & \text{ in }\Omega_T,\\
				v= g  & \text{ in } (\Omega_e)_T,\\
				v(0) =0 & \text{ in } \Omega,
			\end{cases}
		\end{align}
		under suitable regularity assumptions of $Q$, $g$ and $G$. Here the potential $Q$ may not be of the same form as the function $Q_\gamma$ given by the spacetime Liouville reduction \eqref{eq: space-time potential}.
	\end{remark}

	\section{Nonlocal Neumann derivative and new DN maps}
	\label{sec: new DN maps}

	Motivated by \cite[Lemma~3.3]{XavierNeumannBdry}, we define for a given function $u$ the analogous \emph{nonlocal normal derivative} in the exterior domain by 
	\begin{align}
		\mathcal{N}_{\gamma}u(x,t)=C_{n,s} \int_{\Omega} \gamma^{1/2}(x,t)\gamma^{1/2}(y,t) \frac{u(x,t)-u(y,t)}{|x-y|^{n+2s}} \, dy, \quad (x,t)\in (\Omega_e)_T,
	\end{align}
	where $C_{n,s}$ is the constant given by \eqref{C_ns}.

	\subsection{Alternative definition of the DN map}
	Let us make a new definition of the DN map:

	\begin{definition}[New DN map for nonlocal diffusion equation]\label{def: DN maps}
		Let $\Omega\subset \R^n$ be an open set bounded in one direction, $0<T<\infty$, $0<s<\min(1,n/2)$, $\gamma_0>0$ and $\gamma\in \Gamma_{s,\gamma_0}(\R^n_T)$. Then we define the exterior DN map $\mathcal{N}_{\gamma}$ by
		\begin{equation}\label{exterior DN}
			\begin{split}
				\left\langle\mathcal{N}_{\gamma}f, g \right\rangle=&\frac{C_{n,s}}{2}\int_0^T\int_{\R^{2n}\setminus(\Omega_e\times \Omega_e)}\gamma^{1/2}(x,t)\gamma^{1/2}(y,t)\\
				& \qquad \qquad  \cdot \frac{(u_f(x,t)-u_f(y,t))(g(x,t)-g(y,t))}{|x-y|^{n+2s}}\,dxdy\,dt,
			\end{split}
		\end{equation}
		for all $f,g\in C_c^{\infty}((\Omega_e)_T)$,
		where $u_f$ is the unique solution  (see Corollary~\ref{cor: well-posedness} for the well-posedness) of
		\begin{align}\label{eq:u_f}
			\begin{cases}
				\partial_tu+\Div_s(\Theta_{\gamma}\nabla^s u)= 0 &\text{ in }\Omega_T,\\
				u= f &  \text{ in } (\Omega_e)_T,\\
				u(0)=0 &\text{ in } \Omega.
			\end{cases}
		\end{align}
	\end{definition}
	
	\begin{proposition}
		\label{prop: alternative def of DN map cond eq}
		Let $\Omega\subset \R^n$ be an open set bounded in one direction, $0<T<\infty$, $0<s<\min(1,n/2)$, $\gamma_0>0$ and $\gamma\in \Gamma_{s,\gamma_0}(\R^n_T)$. Let $f,g\in C_c^{\infty}((\Omega_e)_T)$, denote by $u_f\in L^2(0,T\,;H^s(\R^n))$ the unique solutions to 
		\begin{align}\label{eq: prop DN map}
			\begin{cases}
				\partial_tu+\Div_s(\Theta_{\gamma}\nabla^s u)= 0 &\text{ in }\Omega_T,\\
				u= f &  \text{ in } (\Omega_e)_T,\\
				u(0)=0 &\text{ in } \Omega.
			\end{cases}
		\end{align}
		Let $v_g\in L^2(0,T\,;H^s(\R^n))$ be any function satisfying $\partial_t v_g \in L^2(0,T\, ; H^{-s}(\Omega))$, and $v_g-g\in L^2(0,T\,;\widetilde{H}^s(\Omega))$, then there holds
		\begin{equation}
			\label{eq: alternative def DN map}
			\begin{split}
				\langle\mathcal{N}_{\gamma}f,g\rangle & = \int_{\Omega_T}\partial_tu_fv_g\,dxdt+\int_{0}^TB_{\gamma}(u_f,v_g)\,dt\\
				&\quad -\frac{C_{n,s}}{2}\int_0^T\int_{\Omega_e\times \Omega_e}\gamma^{1/2}(x,t)\gamma^{1/2}(y,t)\\
				& \qquad \qquad \cdot\frac{(f(x,t)-f(y,t))(g(x,t)-g(y,t))}{|x-y|^{n+2s}}\,dxdy\,dt.
			\end{split}
		\end{equation}
	\end{proposition}
	
	\begin{proof}
		By the definition \eqref{eq: conductivity bilinear form} of the bilinear form $B_{\gamma}$, there holds
		\[\begin{split}
			\langle\mathcal{N}_{\gamma}u,g\rangle =&\frac{C_{n,s}}{2}\int_0^T\int_{\R^{2n}\setminus(\Omega_e\times \Omega_e)}\gamma^{1/2}(x,t)\gamma^{1/2}(y,t)\\
			&\qquad \qquad \cdot \frac{(u_f(x,t)-u_f(y,t))(g(x,t)-g(y,t))}{|x-y|^{n+2s}}\,dxdy\,dt\\
			=&\int_0^TB_{\gamma}(u_f,g)\,dt\\
			&-\frac{C_{n,s}}{2}\int_0^T\int_{\Omega_e\times \Omega_e}\gamma^{1/2}(x,t)\gamma^{1/2}(y,t)\\
			&\qquad \qquad \cdot \frac{(u_f(x,t)-u_f(y,t))(g(x,t)-g(y,t))}{|x-y|^{n+2s}}\,dxdy\,dt.
		\end{split} 
		\] 
		First note that by writting $u_f=(u_f-f)+f$ and using $\LC u_f-f\RC (\cdot,t) \in \widetilde{H}^s(\Omega)$ for a.e. $t\in (0,T)$, the last term is equal to
		\[
		-\frac{C_{n,s}}{2}\int_0^T\int_{\Omega_e\times \Omega_e}\gamma^{1/2}(x,t)\gamma^{1/2}(y,t)\frac{(f(x,t)-f(y,t))(g(x,t)-g(y,t))}{|x-y|^{n+2s}}\,dxdy\,dt.
		\]
		On the other hand by using cutoff functions $\eta_m$ in time vanishing near $t=0$ and $t=T$ but equal to one on the support of $g$, one obtains
		\[
		\begin{split}
			\int_0^TB_{\gamma}(u_f,g)\,dt
			=&\lim_{m\to\infty}\int_0^TB_{\gamma}(u_f,\eta_mg)\,dt\\
			=&\lim_{m\to\infty}\left(-\int_0^T B_{\gamma}(u_f,\eta_m(v_g-g))\,dt+\int_0^TB_{\gamma}(u_f,\eta_mv_g)\,dt\right)\\
			=&-\lim_{m\to\infty}\int_{\Omega_T}u_f\partial_t(\eta_m(v_g-g))\,dxdt+\int_{0}^TB_{\gamma}(u_f,v_g)\,dt\\
			=&\lim_{m\to\infty}\int_{\Omega_T}\partial_tu_f \eta_mv_g\,dxdt+\int_0^TB_{\gamma}(u_f,v_g)\,dt\\
			=&\int_{\Omega_T}\partial_tu_f v_g\,dxdt+\int_0^TB_{\gamma}(u_f,v_g)\,dt.
		\end{split}
		\]
		This concludes the proof.
	\end{proof}
	

	We next define the DN map for the spacetime Liouville reduction equation by the corresponding nonlocal Neumann deriative.

	\begin{definition}
		\label{def: DN map schroedinger type equation}
		Let $\Omega\subset \R^n$ be an open set bounded in one direction, $0<T<\infty$, $0<s<\min(1,n/2)$, $\gamma_0>0$ and $\gamma\in \Gamma_{s,\gamma_0}(\R^n_T)$. Then we define the exterior DN map $\mathcal{N}_{Q_{\gamma}}$ by
		\begin{equation}\label{exterior DN schroedinger}
			\begin{split}
				\left\langle\mathcal{N}_{Q_{\gamma}}f, g \right\rangle=\frac{C_{n,s}}{2}\int_0^T\int_{\R^{2n}\setminus(\Omega_e\times \Omega_e)}\frac{(v_f(x,t)-v_f(y,t))(g(x,t)-g(y,t))}{|x-y|^{n+2s}}\,dxdy\,dt,
			\end{split}
		\end{equation}
		for all $f,g\in C_c^{\infty}((\Omega_e)_T)$,
		where $v_f$ is the unique solution (see Theorem~\ref{thm:fractionalLiouvilleReduction} and Proposition \ref{prop: solution schroedinger eqs}) of
		\begin{align}
			\label{eq: uf DN map}
			\begin{cases}
				\partial_t\LC \gamma^{-1}v\RC +\LC (-\Delta)^s+Q_{\gamma}\RC v= 0  & \text{ in }\Omega_T,\\
				v= g  & \text{ in } (\Omega_e)_T,\\
				v(0) =0 & \text{ in } \Omega,
			\end{cases}
		\end{align}
		and $C_{n,s}$ is the constant given by \eqref{C_ns}.
	\end{definition}

	To prove Theorem \ref{Theorem: General formulation}, let us derive a useful representation formula of \eqref{exterior DN schroedinger}. 
	
	\begin{proposition}
		\label{prop: alternative def of DN map schroedinger}
		Let $\Omega\subset \R^n$ be an open set bounded in one direction, $0<T<\infty$, $0<s<\min(1,n/2)$, $\gamma_0>0$ and $\gamma\in \Gamma_{s,\gamma_0}(\R^n_T)$. Let $f,g\in C_c^{\infty}((\Omega_e)_T)$, denote by $u_f$ the unique solutions to 
		\begin{align}
			\label{eq: uf prop DN map}
			\begin{cases}
				\partial_t\LC \gamma^{-1}u\RC +\LC (-\Delta)^s+Q_{\gamma}\RC u= 0  & \text{ in }\Omega_T,\\
				u= f  & \text{ in } (\Omega_e)_T,\\
				u(0) =0 & \text{ in } \Omega.
			\end{cases}
		\end{align}
		Let $v_g\in L^2(0,T\,;H^s(\R^n))$ be any function satisfying $\partial_t v_g\in L^2(0,T,\, H^{-s}(\Omega))$, and $v_g-g\in L^2(0,T\,;\widetilde{H}^s(\Omega))$, then there holds
		\begin{equation}
			\label{eq: alternative def DN map1}
			\begin{split}
				&\langle\mathcal{N}_{Q_{\gamma}}f,g\rangle \\
				=& \int_{\Omega_T}\partial_t(\gamma^{-1}u_f)v_g\,dxdt+\int_{\R^n_T}(-\Delta)^{s/2}u_f(-\Delta)^{s/2}v_g\,dxdt+\int_{\Omega_T}Q_{\gamma}u_fv_g\,dxdt\\
				& -\frac{C_{n,s}}{2}\int_0^T\int_{\Omega_e\times \Omega_e}\frac{(f(x,t)-f(y,t))(g(x,t)-g(y,t))}{|x-y|^{n+2s}}\,dxdy\,dt.
			\end{split}
		\end{equation}
	\end{proposition}
	
	\begin{remark}
		Observe that the last term in \eqref{eq: alternative def DN map} is independent of $Q_{\gamma}$. Therefore, in this case the corresponding DN map $\Lambda_{Q_\gamma}$ can be defined by
		\[
		\begin{split}
			&\langle \Lambda_{Q_\gamma}f,g\rangle\\
			=&\int_{\Omega_T}\partial_t(\gamma^{-1}u_f)v_g\,dxdt+\int_{\R^n_T}(-\Delta)^{s/2}u_f(-\Delta)^{s/2}v_g\,dxdt+\int_{\Omega_T}Q_{\gamma}u_fv_g\, dxdt
		\end{split}
		\]
		for $f,g\in C_c^{\infty}((\Omega_e)_T)$ contains the same information as $\mathcal{N}_{Q_{\gamma}}$.
	\end{remark}
	
	\begin{proof}[Proof of Proposition \ref{prop: alternative def of DN map schroedinger}]
		As in the proof of Proposition~\ref{prop: alternative def of DN map cond eq} there holds
		\[
		\begin{split}
			\langle\mathcal{N}_{Q_\gamma}f,g\rangle & = \int_{\R^n_T}(-\Delta)^{s/2}u_f(-\Delta)^{s/2}g\,dxdt\\
			&\quad-\frac{C_{n,s}}{2}\int_0^T\int_{\Omega_e\times \Omega_e}\frac{(f(x,t)-f(y,t))(g(x,t)-g(y,t))}{|x-y|^{n+2s}}\,dxdy\,dt.
		\end{split}
		\]
		Next, as in  the proof of Proposition~\ref{prop: alternative def of DN map cond eq}, we use a sequence of cutoff functions $(\eta_m)_{m\in\N}\subset C_c^{\infty}((0,T))$ to deduce the identity
		\[
		\begin{split}
			&\int_{\R^n_T}(-\Delta)^{s/2}u_f(-\Delta)^{s/2}g\,dxdt\\
			=&\lim_{m\to\infty}-\int_{\R^n_T}(-\Delta)^{s/2}u_f(-\Delta)^{s/2}(\eta_m(v_g-g))\,dxdt\\
			&+\lim_{m\to\infty}\int_{\R^n_T}(-\Delta)^{s/2}u_f(-\Delta)^{s/2}(\eta_mv_g)\,dxdt\\
			=&\lim_{m\to\infty}\left(-\int_{\Omega_T}\gamma^{-1}u_f\partial_t(\eta_m(v_g-g))\,dxdt+\int_{\Omega_T}Q_{\gamma}u_f(\eta_m(v_g-g))\,dxdt\right)\\
			&+\int_{\R^n_T}(-\Delta)^{s/2}u_f(-\Delta)^{s/2}v_g\,dxdt\\
			=&\int_{\Omega_T}\partial_t(\gamma^{-1}u_f)v_g\,dxdt+\int_{\R^n_T}(-\Delta)^{s/2}u_f(-\Delta)^{s/2}v_g\,dxdt+\int_{\Omega_T}Q_{\gamma}u_fv_g\,dxdt.
		\end{split}
		\]
		This completes the proof.
	\end{proof}

	\subsection{Relation between DN map and nonlocal Neumann derivative}
	
	Let us consider two arbitrary nonempty open subsets $W_1 ,W_2 \subset W$, with $W_1\cap W_2 =\emptyset$, where $W\subset \Omega_e$ denotes the open set in the statements of either Theorem \ref{Theorem: General formulation} or Theorem \ref{Theorem: Exterior determination}. Meanwhile, with the exterior determination result (Theorem \ref{Theorem: Exterior determination}) at hand, one already knows that $\gamma_1=\gamma_2$ in $W_T$, provided that $\left. \Lambda_{\gamma_1}f \right|_{W_T}=\left. \Lambda_{\gamma_2}f \right|_{W_T}$ for any $f\in C^\infty_c(W_T)$. Adopting these notations, one immediately has $\gamma_1=\gamma_2$ in $(W_1 \cup W_2)_T$. Then we can derive the following relation.

	\begin{lemma}\label{Lemma: old DN implies new DN}
		Let $\Omega\subset \R^n$ be an open set bounded in one direction, $0<T<\infty$, $0<s<\min(1,n/2)$, $\gamma_0>0$ and $\gamma_j \in \Gamma_{s,\gamma_0}(\R^n_T)$ for $j=1,2$. Assume that $W_1,W_2\subset \Omega_e$ are two nonempty open disjoint sets and $\Gamma\in \Gamma_{s,\gamma_0}(\R^n_T)$ are such that $\gamma_1(x,t)=\gamma_2(x,t)=\Gamma(x,t)$ for all $(x,t)\in (W_1\cup W_2)_T$. Then we have 
		$$\left. \Lambda_{\gamma_1}f \right|_{(W_2)_T}=\left. \Lambda_{\gamma_2}f \right|_{(W_2)_T} \quad \text{for any } f \in C^\infty_c((W_1)_T)
		$$ 
		if and only if there holds
		\begin{align}
			\langle\mathcal{N}_{\gamma_1}f,g\rangle =\langle
			\mathcal{N}_{\gamma_2}f,g\rangle \quad\text{for any }f \in C^{\infty}_c((W_1)_T) \text{ and }g\in C^\infty_c((W_2)_T).
		\end{align}
	\end{lemma}
	
	\begin{proof}
		We have for any $f\in C^\infty_c((W_1)_T)$ and $g\in C_c^{\infty}((W_2)_T)$
		\begin{equation}\label{eq: DN map diffusion}
			\langle \Lambda_{\gamma}f,g\rangle=\frac{C_{n,s}}{2}\int_0^T\int_{\R^{2n}}\gamma^{1/2}(x,t)\gamma^{1/2}(y,t)\frac{(u_f(x,t)-u_f(y,t))(g(x,t)-g(y,t))}{|x-y|^{n+2s}}\,dxdydt
		\end{equation}
		by using Definition \ref{def: the DN map}.
		Thus, combining with Definition \ref{def: DN maps}, one has that 
		\begin{equation}\label{eq: relation different DN maps diffusion}
			\begin{split}
				\langle \Lambda_{\gamma_1}f,g\rangle =&\langle\mathcal{N}_{\gamma_1}f,g\rangle+\frac{C_{n,s}}{2}\int_0^T\int_{\Omega_e\times\Omega_e}\gamma_1^{1/2}(x,t)\gamma_1^{1/2}(y,t)\\
				& \qquad \qquad \qquad  \cdot \frac{(f(x,t)-f(y,t))(g(x,t)-g(y,t))}{|x-y|^{n+2s}}\,dxdy\,dt\\
				=&\langle\mathcal{N}_{\gamma_2}f,g\rangle+\frac{C_{n,s}}{2}\int_0^T\int_{\Omega_e\times\Omega_e}\gamma_2^{1/2}(x,t)\gamma_2^{1/2}(y,t)\\
				& \qquad \qquad \qquad  \cdot \frac{(f(x,t)-f(y,t))(g(x,t)-g(y,t))}{|x-y|^{n+2s}}\,dxdy\,dt\\
				=&\langle \Lambda_{\gamma_2}f,g\rangle,
			\end{split}
		\end{equation}
		where we used that $u_f=f$ in $(\Omega_e)_T$.
		
		On the other hand, one can see that
		\begin{align}\label{eq: relation different DN maps diffusion 1}
			\begin{split}
				&\int_0^T\int_{\Omega_e\times\Omega_e}\gamma_1^{1/2}(x,t)\gamma_1^{1/2}(y,t) \frac{(f(x,t)-f(y,t))(g(x,t)-g(y,t))}{|x-y|^{n+2s}}\,dxdy\,dt \\
				=&\int_0^T\int_{\Omega_e\times\Omega_e}\gamma_2^{1/2}(x,t)\gamma_2^{1/2}(y,t) \frac{(f(x,t)-f(y,t))(g(x,t)-g(y,t))}{|x-y|^{n+2s}}\,dxdy\,dt ,
			\end{split}
		\end{align}
		where we used that $f\in C^\infty_c((W_1)_T)$, $g\in C_c^{\infty}((W_2)_T)$ with $\gamma_1=\gamma_2$ in $(W_1\cup W_2)_T$ and $W_1\cap W_2=\emptyset$. Finally, insert \eqref{eq: relation different DN maps diffusion 1} into \eqref{eq: relation different DN maps diffusion}, we can see the assertion is true. This completes the proof.
	\end{proof}

	\begin{theorem}[Relation of DN maps]\label{Thm: relation of DNs}
		\label{eq: relation DN maps}
		Let $\Omega\subset \R^n$ be an open set bounded in one direction, $0<T<\infty$, $0<s<\min(1,n/2)$, $\gamma_0>0$ and $\gamma_j \in \Gamma_{s,\gamma_0}(\R^n_T)$ for $j=1,2$. Assume that $W_1,W_2\subset \Omega_e$ are two nonempty open disjoint sets and $\Gamma\in \Gamma_{s,\gamma_0}(\R^n_T)$ are such that $\gamma_1(x,t)=\gamma_2(x,t)=\Gamma(x,t)$ for all $(x,t)\in (W_1\cup W_2)_T$ and $\Gamma\in C^{\infty}((W_1\cup W_2)_T)$. Then for $f\in C_c^{\infty}((W_1)_T),g\in C_c^{\infty}((W_2)_T)$, we have
		\[
		\langle\mathcal{N}_{\gamma_1}f,g\rangle =\langle
		\mathcal{N}_{\gamma_2}f,g\rangle
		\]
		if and only if
		\[
		\langle\mathcal{N}_{Q_{\gamma_1}}(\Gamma^{1/2}f),(\Gamma^{1/2}g)\rangle = \langle\mathcal{N}_{Q_{\gamma_2}}(\Gamma^{1/2}f),(\Gamma^{1/2}g)\rangle.
		\]
	\end{theorem}
	\begin{proof}
		By the Liouville reduction (cf.~Theorem~\ref{thm:fractionalLiouvilleReduction}) there holds
		\[
		\begin{split}
			&\int_{\Omega_T}\partial_tu_fv_g\,dxdt+\int_{0}^TB_{\gamma}(u_f,v_g)\,dt\\
			=&\int_{\Omega_T}\partial_t(\gamma^{-1}(\gamma^{1/2}u_f))(\gamma^{1/2}v_g)\,dxdt+\int_{\R^n_T}(-\Delta)^{s/2}(\gamma^{1/2}u_f)(-\Delta)^{s/2}(\gamma^{1/2}v_g)\,dxdt\\
			& +\int_{\R^n_T}q_{\gamma}(\gamma^{1/2}u_f)(\gamma^{1/2}v_g)\,dxdt-\int_{\Omega_T}\frac{1}{2\gamma^2}(\gamma^{1/2}u_f)(\gamma^{1/2}v_g)\,dxdt\\
			=&\int_{\Omega_T}\partial_t(\gamma^{-1}(\gamma^{1/2}u_f))(\gamma^{1/2}v_g)\,dxdt+\int_{\R^n_T}(-\Delta)^{s/2}(\gamma^{1/2}u_f)(-\Delta)^{s/2}(\gamma^{1/2}v_g)\,dxdt\\
			& +\int_{\Omega_T}Q_{\gamma}(\gamma^{1/2}u_f)(\gamma^{1/2}v_g)\,dxdt +\int_{(\Omega_e)_T}q_{\gamma}(\gamma^{1/2}u_f)(\gamma^{1/2}v_g)\,dxdt\\
			=&\langle \mathcal{N}_{Q_{\gamma}}(\Gamma^{1/2}f),(\Gamma^{1/2}g)\rangle +\int_{(\Omega_e)_T}q_{\gamma}(\gamma^{1/2}f)(\gamma^{1/2}g)\,dxdt\\
			& +\frac{C_{n,s}}{2}\int_0^T\int_{\Omega_e\times \Omega_e}\frac{(f(x,t)-f(y,t))(g(x,t)-g(y,t))}{|x-y|^{n+2s}}\,dxdy\,dt,
		\end{split}
		\]
		since $w_f\vcentcolon = \gamma^{1/2}u_f$ solves \eqref{eq: uf prop DN map} with exterior condition $\Gamma^{1/2}f$ and $v_{\Gamma^{1/2}g}\vcentcolon = \gamma^{1/2}v_g$ is an extension of $\gamma^{1/2}g$ with the same regularity properties as $v_g$. Therefore, there holds
		\[
		\begin{split}
			&\langle \mathcal{N}_{\gamma}f,g\rangle-\langle\mathcal{N}_{Q_{\gamma}}(\Gamma^{1/2}f),(\Gamma^{1/2}g)\rangle \\
			=&\int_{(\Omega_e)_T}q_{\gamma}(\gamma^{1/2}f)(\gamma^{1/2}g)\,dxdt\\
			&+ \frac{C_{n,s}}{2}\int_0^T\int_{\Omega_e\times \Omega_e}\frac{(f(x,t)-f(y,t))(g(x,t)-g(y,t))}{|x-y|^{n+2s}}\,dxdy\,dt\\
			&+\frac{C_{n,s}}{2}\int_0^T\int_{\Omega_e\times \Omega_e}\gamma^{1/2}(x,t)\gamma^{1/2}(y,t)\\ &\qquad \qquad \qquad \cdot \frac{(f(x,t)-f(y,t))(g(x,t)-g(y,t))}{|x-y|^{n+2s}}\,dxdy\,dt.
		\end{split}
		\]
		Since $W_1\cap W_2=\emptyset$, it follows that
		\[
		\begin{split}
			&\langle \mathcal{N}_{\gamma}f,g\rangle-\langle\mathcal{N}_{Q_{\gamma}}(\Gamma^{1/2}f),(\Gamma^{1/2}g)\rangle\\
			=&-\frac{C_{n,s}}{2}\int_0^T\int_{\Omega_e\times \Omega_e}\LC 1+\Gamma^{1/2}(x,t)\Gamma^{1/2}(y,t)\RC \\
			&\qquad \qquad \cdot \frac{(f(x,t)g(y,t)-f(y,t)g(x,t))}{|x-y|^{n+2s}}\,dxdy\,dt.
		\end{split}
		\]
		Now this expression on the right hand side does not depend on the conductivities and so we see that there holds 
		$$
		\langle\mathcal{N}_{\gamma_1}f,g\rangle =\langle \mathcal{N}_{\gamma_2}f,g\rangle \text{ if and only if } \langle\mathcal{N}_{Q_{\gamma_1}}f,g\rangle=\langle\mathcal{N}_{Q_{\gamma_2}}f,g\rangle,
		$$ 
		for any $f\in C^\infty_c((W_1)_T)$ and $g\in C^\infty_c((W_2)_T)$.
		This proves the assertion.
	\end{proof}

	\subsection{Adjoint DN map}
	
	Let us introduce the adjoint DN map which then will be used to prove a suitable integral identity.
	
	\begin{definition}[Adjoint DN map]
		\label{def: adjoint DN map}
		Let $\Omega\subset \R^n$ be an open set bounded in one direction, $0<T<\infty$, $0<s<\min(1,n/2)$, $\gamma_0>0$ and $\gamma\in \Gamma_{s,\gamma_0}(\R^n_T)$. Then we define the adjoint exterior DN map $\mathcal{N}^{\ast}_{Q_{\gamma}}$ by
		\begin{equation}
			\label{eq: adjoint DN map}
			\begin{split}
				\langle \mathcal{N}^{\ast}_{Q_{\gamma}}f,g\rangle
				=\frac{C_{n,s}}{2}\int_0^T\int_{\R^{2n}\setminus(\Omega_e\times \Omega_e)}\frac{(u_f(x,t)-u_f(y,t))(g(x,t)-g(y,t))}{|x-y|^{n+2s}}\,dxdy\,dt
			\end{split}
		\end{equation}
		for all $f,g\in C_c^{\infty}((\Omega_e)_T)$, where $u_f$ is the unique solution to 
		\begin{align}
			\label{eq: adjoint DN map Schroedinger}
			\begin{cases}
				-\gamma^{-1}\partial_tv +\LC (-\Delta)^s+Q_{\gamma}\RC v= 0  & \text{ in }\Omega_T,\\
				v= f  & \text{ in } (\Omega_e)_T,\\
				v(T) =0 & \text{ in } \Omega,
			\end{cases}
		\end{align}
	and $C_{n,s}$ is the constant given by \eqref{C_ns}.
	\end{definition}

	We make the following simple observations:
	
	\begin{lemma}[Properties adjoint DN map]
		\label{lemma: properties adjoint DN map}
		Let $\Omega\subset \R^n$ be an open set bounded in one direction, $0<T<\infty$, $0<s<\min(1,n/2)$, $\gamma_0>0$ and $\gamma\in \Gamma_{s,\gamma_0}(\R^n_T)$. If $f,g \in C_c^{\infty}((\Omega_e)_T)$ and $u_f$ is the unique solution to 
		\begin{align}
			\label{eq: sol prop adjoint DN map}
			\begin{cases}
				-\gamma^{-1}\partial_tu +\LC (-\Delta)^s+Q_{\gamma}\RC u= 0  & \text{ in }\Omega_T,\\
				u= f  & \text{ in } (\Omega_e)_T,\\
				u(T) =0 & \text{ in } \Omega
			\end{cases}
		\end{align}
		then
		\begin{enumerate}[(i)]
			\item\label{item alternative representation of adjoint DN map}
			for any any extension $v_g$ of $g$ with $v_g\in L^2(0,T\,;H^s(\R^n))$ and $\partial_tv_g\in L^2(0,T,\, H^{-s}(\Omega))$ there holds
			\begin{equation}
				\label{eq: alternative rep adj DN map}
				\begin{split}
					\langle &\mathcal{N}^{\ast}_{Q_{\gamma}}f,g\rangle\\
					=&-\int_{\Omega_T}\gamma^{-1}\partial_t u_fv_g\,dxdt+\int_{\R^n_T}(-\Delta)^{s/2}u_f(-\Delta)^{s/2}v_g\,dxdt+\int_{\Omega_T}Q_{\gamma}u_fv_g\,dxdt\\
					& -\frac{C_{n,s}}{2}\int_0^T\int_{\Omega_e\times \Omega_e}\frac{(f(x,t)-f(y,t))(g(x,t)-g(y,t))}{|x-y|^{n+2s}}\,dxdy\,dt,
				\end{split}
			\end{equation}
			\item\label{item selfadjointness of DN map} 
			\begin{equation}
				\label{eq: selfadjointness DN map}
				\langle \mathcal{N}^{\ast}_{Q_{\gamma}}f,g\rangle = \langle\mathcal{N}_{Q_{\gamma}}g,f \rangle.
			\end{equation}
		\end{enumerate}
	\end{lemma}
	
	\begin{proof}
		(i): This follows from a similar calculation as in Proposition~\ref{prop: alternative def of DN map schroedinger}.\\
		
		\noindent (ii): 
	Let $u_g$ be the solution to \eqref{eq: uf prop DN map} as the exterior data $f$ is replaced by $g$.
	Since $u_f(T)=0$, $u_g(0)=0$ there holds
	\[
	\int_{\Omega_T}\LC \partial_t(\gamma^{-1}u_g)u_f+u_g\gamma^{-1}\partial_tu_f) \RC dxdt=0
	\]
	This immediately shows the claim.
\end{proof}

\section{The global uniqueness}

We split this final section into several parts. We first establish the integral identity and the Runge approximations in Sections \ref{sec: integral id} and \ref{sec: runge approx.}, respectively. Combined with these two statements, one can prove the interior uniqueness in Section \ref{sec: interior det}. Finally, we show the UCP of exterior DN maps, which together with the work of Section \ref{sec: exterior det} imply the global uniqueness result of Theorem \ref{Theorem: General formulation}. 

\subsection{Integral identity}\label{sec: integral id}

One of the key material to prove the interior uniqueness is to derive a suitable integral identity.

\begin{proposition}[Integral identity]
	\label{prop: integral identity}
	Let $\Omega\subset \R^n$ be an open set bounded in one direction, $0<T<\infty$, $0<s<\min(1,n/2)$, $\gamma_0>0$ and $\gamma_j\in \Gamma_{s,\gamma_0}(\R^n_T)$ for $j=1,2$. Assume that $W_1,W_2\subset \Omega_e$ are two nonempty open sets and $\Gamma\in \Gamma_{s,\gamma_0}(\R^n_T)$ are such that $\gamma_1(x,t)=\gamma_2(x,t)=\Gamma(x,t)$ for all $(x,t)\in (W_1\cup W_2)_T$ and $\Gamma\in C^{\infty}((W_1\cap W_2)_T)$. Then for $f\in C_c^{\infty}((W_1)_T),g\in C_c^{\infty}((W_2)_T)$ we have
	\begin{equation}
		\label{eq: difference schroedinger DN maps}
		\langle(\mathcal{N}_{Q_{\gamma_1}}-\mathcal{N}_{Q_{\gamma_2}})f,g\rangle = \int_{\Omega_T}(\gamma_2^{-1}-\gamma_1^{-1})v_f\partial_tv_g\,dxdt+\int_{\Omega_T}(Q_{\gamma_1}-Q_{\gamma_2})v_gv_f\,dxdt,
	\end{equation}
	where $v_f$ is the unique solutions to \eqref{eq: uf prop DN map} with $\gamma=\gamma_1$ and $v_g$ is the unique solution to the adjoint equation
	\begin{align}
		\label{eq: adjoint sol}
		\begin{cases}
			-\gamma^{-1}_2\partial_t w +\LC (-\Delta)^s+Q_{\gamma_2}\RC w= 0  & \text{ in }\Omega_T,\\
			w= g  & \text{ in } (\Omega_e)_T,\\
			w(T) =0 & \text{ in } \Omega.
		\end{cases}
	\end{align}
\end{proposition}

\begin{proof}
	By Lemma~\ref{lemma: properties adjoint DN map}, Proposition~\ref{prop: alternative def of DN map schroedinger} we have
	\[
	\begin{split}
		&\langle(\mathcal{N}_{Q_{\gamma_1}}-\mathcal{N}_{Q_{\gamma_2}})f,g\rangle \\
		=&\langle \mathcal{N}_{Q_{\gamma_1}}f,g\rangle -\langle \mathcal{N}_{Q_{\gamma_2}}f,g\rangle\\
		=&\langle \mathcal{N}_{Q_{\gamma_1}}f,g\rangle -\langle \mathcal{N}^{\ast}_{Q_{\gamma_2}}g,f\rangle\\
		=&\int_{\Omega_T}\partial_t(\gamma_1^{-1}v_f)v_g\,dxdt+\int_{\R^n_T}(-\Delta)^{s/2}v_f(-\Delta)^{s/2}v_g\,dxdt+\int_{\Omega_T}Q_{\gamma_1}v_fv_g\,dxdt\\
		& -\frac{C_{n,s}}{2}\int_0^T\int_{\Omega_e\times \Omega_e}\frac{(f(x,t)-f(y,t))(g(x,t)-g(y,t))}{|x-y|^{n+2s}}\,dxdy\,dt\\
		&+\int_{\Omega_T}\gamma_2^{-1}\partial_t v_gv_f\,dxdt-\int_{\R^n_T}(-\Delta)^{s/2}v_g(-\Delta)^{s/2}v_f\,dxdt-\int_{\Omega_T}Q_{\gamma_2}v_gv_f\,dxdt\\
		& +\frac{C_{n,s}}{2}\int_0^T\int_{\Omega_e\times \Omega_e}\frac{(f(x,t)-f(y,t))(g(x,t)-g(y,t))}{|x-y|^{n+2s}}\,dxdy\,dt\\
		=&\int_{\Omega_T}\partial_t(\gamma_1^{-1}v_f)v_g\,dxdt+\int_{\Omega_T}\gamma_2^{-1}v_f\partial_tv_g\,dxdt+\int_{\Omega_T}(Q_{\gamma_1}-Q_{\gamma_2})v_gv_f\,dxdt\\
		=&\int_{\Omega_T}(\gamma_2^{-1}-\gamma_1^{-1})v_f\partial_tv_g\,dxdt+\int_{\Omega_T}(Q_{\gamma_1}-Q_{\gamma_2})v_gv_f\,dxdt\\
	\end{split}
	\]
	where we used for the integration by parts that $v_f(0)=0$ and $v_g(T)=0$.
\end{proof}

\subsection{Approximation property}\label{sec: runge approx.}

To prove the interior uniqueness result of $\gamma$, we derive an approximation property of solutions to the Schr\"odinger type equations. First we introduce the source to solution map, which is ususally called Poisson operator.
Assume that $\Omega\subset \R^n$ is an open set bounded in one direction, $0<T<\infty$, $0<s<\min(1,n/2)$, $\gamma_0>0$ and $\gamma\in \Gamma_{s,\gamma_0}(\R^n_T)$ and $W\subset \Omega_e$ is a nonempty open set. With the well-posedness of \eqref{eq: weak solutions of Schrodinger equation existence}, we can define the Poisson operator $P_{\gamma}$ such that 
\begin{align}\label{Poisson op}
	P_{\gamma}: C^\infty_c(W_T)\to  H^1(0,T\,;L^2(\Omega))\cap L^2(0,T\,;H^s(\R^n)), \quad  f\mapsto v_f,
\end{align}
where $v_f\in H^1(0,T\,;L^2(\Omega))\cap L^2(0,T\,;H^s(\R^n))$ is the unique solution of
\begin{align}\label{runge equ}
	\begin{cases}
		\partial_t\LC \gamma^{-1}v\RC +\LC (-\Delta)^s+Q_{\gamma}\RC v=0 & \text{ in } \Omega_T,\\
		v=f &\text{ in }(\Omega_e)_T,\\
		v(0)=0 &\text{ in }\Omega,
	\end{cases}
\end{align}
with $v_f-f\in  H^1(0,T\,;L^2(\Omega))\cap L^2(0,T\,; \wt H^s(\Omega))$.

Next before studying the Runge approximation for the equation \eqref{runge equ}, let us recall the UCP for the fractional Laplacian (see e.g.~ \cite[Theorem 1.2]{GSU20} for functions in $H^r$ or \cite[Theorem~2.2]{KRZ2022Biharm} in $H^{r,p}$).

\begin{proposition}[Unique continuation for the fractional Laplacian]\label{Prop: UCP}
	For $n\in \N$ and $s\in (0,1)$, let $w\in H^{-r}(\R^n)$ for some $r\in \R$. Given a nonempty open subset $W\subset \R^n$, then 
	$w=(-\Delta)^s w=0$ in $W$ implies that $w\equiv 0$ in $\R^n$.
\end{proposition}

\begin{proposition}[Runge approximation]\label{Prop: runge}
	Let $\Omega\subset \R^n$ be an open set bounded in one direction, $0<T<\infty$, $0<s<\min(1,n/2)$, $\gamma_0>0$ and $\gamma\in \Gamma_{s,\gamma_0}(\R^n_T)$ and $W\subset \Omega_e$ be a nonempty open set. Let $P_\gamma$ be the Poisson operator given by \eqref{Poisson op}, and 
	\begin{align}
		\mathcal{R}\vcentcolon =\left\{v_f-f\,; \, f\in C_c^{\infty}(W_T)\,\right\}.
	\end{align}
	Then the set $\mathcal{R}$ is dense in $L^2(0,T\,; \wt H^s(\Omega))$.
\end{proposition}

\begin{proof}
	By using Theorem \ref{thm:fractionalLiouvilleReduction}, one has $\mathcal{R}\subset   L^2(0,T;\, \wt H^s(\Omega))$. In order to show the density, by the Hahn-Banach theorem, we only need to show that for any $F\in (L^2(0,T;\, \wt H^s(\Omega)))^\ast =L^2(0,T;\, H^{-s}(\Omega))$,
	such that $\langle F, w\rangle=0$\footnote{Here $\langle F, w\rangle=\int_{\Omega_T} F w \, dxdt$ denotes the duality pairing, for $F\in L^2(0,T;\,  H^{-s}(\Omega))$ and $w\in L^2(0,T;\, \wt H^s(\Omega))$.},
	for any $w\in \mathcal{R}$, then $F$ must be zero. 
	Via $\langle F, w\rangle=0$ for any $w\in \mathcal{R}$, we have 
	\begin{align}
		\langle F,  P_{\gamma}f-f\rangle=0, \text{ for any } f\in C^\infty_c (W_T).
	\end{align}

	We next claim that 
	\begin{align}\label{pair-integ}
		\langle F,  P_{\gamma}f-f\rangle=-\int_{\R^n}(-\Delta)^{s/2} f (-\Delta)^{s/2} \varphi \, dxdt, \text{ for any }f\in C^\infty_c(W_T),
	\end{align}
	where $\varphi\in L^2(0,T;\, H^s(\R^n))$ with $\p _t \varphi \in  L^2(0,T;\, H^{-s}(\Omega))$ (see Proposition \ref{prop: solution schroedinger eqs}) is the unique solution of the adjoint equation 
	\begin{align}
		\begin{cases}
			-\gamma^{-1}\partial_t \varphi +\LC (-\Delta)^s+Q_{\gamma}\RC \varphi= F  & \text{ in }\Omega_T,\\
			\varphi= 0  & \text{ in } (\Omega_e)_T,\\
			\varphi(T) =0 & \text{ in } \Omega.
		\end{cases}
	\end{align}
	In fact, by direct computations, one has that $v_f = P_{\gamma}f$ and 
	\begin{align}
		\begin{split}
			\langle F,  P_{\gamma}f-f\rangle = &\int_{\Omega_T} \LC -\gamma^{-1}\partial_t \varphi +Q_{\gamma}\RC\LC v_f-f\RC\,dxdt \\
			&+\int_{\R^n_T} (-\Delta)^{s/2}\varphi\,  (-\Delta)^{s/2}\LC v_f-f\RC  dxdt\\
			= &\int_{\Omega_T}\LC -\gamma^{-1}\partial_t \varphi +Q_{\gamma}\varphi \RC    v_f \, dxdt +\int_{\R^n_T} (-\Delta)^{s/2} \varphi \,  (-\Delta)^{s/2} v_f dxdt  \\
			& -\int_{\R^n_T} (-\Delta)^{s/2} \varphi \,  (-\Delta)^{s/2} f dxdt\\
			=& \underbrace{\int_{\Omega_T}  \LC \partial_t(\gamma^{-1}v_f)  + Q_\gamma v_f \RC \varphi \, dxdt+\int_{\R^n_T}(-\Delta)^{s/2} v_f\, (-\Delta)^{s/2}\varphi\,dxdt}_{=0\text{, since }v_f=P_{\gamma}f} \\
			&-\int_{\R^n_T}(-\Delta)^{s/2} \varphi  \, (-\Delta)^{s/2} f \, dxdt \\
			=&-\int_{\R^n_T}(-\Delta)^{s/2} \varphi \, (-\Delta)^{s/2} f \, dxdt,
		\end{split}
	\end{align}
	where we used that $v_f(0)=\varphi(T)=0$ for the integration by parts in the third equality sign and the integral involving the time derivative has to be understood in a weak sense. This shows the identity \eqref{pair-integ}. Finally, the identity \eqref{pair-integ} is equivalent to 
	\begin{align}
		(-\Delta)^s\varphi=0\quad\text{in}\quad W_T.
	\end{align}
	Thus, the function $\varphi$ satisfies $\varphi =(-\Delta)^s \varphi =0 $ in $W_T$, by Proposition \ref{Prop: UCP}, then we have $\varphi \equiv 0$ in $\R^n_T$, so that $F\equiv 0$ in $\R^n_T$. In summary, we showed that the set $\mathcal{R}$ is dense in $L^2(0,T;\, \wt H^s(\Omega))$. This proves the assertion.
\end{proof}

\begin{remark}\label{Rmk: runge}
	Note that 
	\begin{itemize}
		\item[(i)] By Proposition \ref{Prop: runge}, we know that given any $\phi \in L^2(0,T;\, \wt H^s(\Omega))$, there exists a sequence of solutions $\{v_{f_k}\}_{k\in \N}\in L^2(0,T;\, H^s(\R^n))$ to \eqref{runge equ} with $f=f_k$, such that 
		$$
		v_{f_k}-f_k \to \phi \text{ in } L^2(0,T;\, \wt H^s(\Omega)) \text{ as } k\to \infty.
		$$
		Since $v_{f_k}$ is a solution, by applying Proposition \ref{prop: solution schroedinger eqs}, we can see that $\p_t v_{f_k}\in L^2(0,T;\, H^{-s}(\Omega))$. Now assume that the time derivative of $\phi$ belongs to $L^2(0,T\,;H^{-s}(\Omega))$, then we have
		\begin{align}
			\lim_{k\to \infty}\int_{\Omega_T} \p_t (v_{f_k}-f_k) \varphi \, dxdt =&-\lim_{k\to \infty}\int_{\Omega_T} (v_{f_k}-f_k) \p_t \varphi \, dxdt \\
			=& -\int_{\Omega_T} \phi \p_t \varphi \, dxdt\\
			=&\int_{\Omega_T} (\p_t\phi)  \varphi \, dxdt,
		\end{align}
		for any $\varphi \in L^2(0,T;\, \wt H^s(\Omega))$ with $\p_t \varphi \in L^2(0,T;\, H^{-s}(\Omega))$ and $\varphi(T)=0$. 
		
		\item[(ii)] By using similar arguments as in the proof of Proposition \ref{Prop: runge}, one can show that the Runge approximation holds for the adjoint diffusion equation 
		\begin{align}\label{runge adj equ}
			\begin{cases}
				-\gamma^{-1}\partial_tv^\ast  +\LC (-\Delta)^s+Q_{\gamma}\RC v^\ast =0 & \text{ in } \Omega_T,\\
				v^\ast=g &\text{ in }(\Omega_e)_T,\\
				v^\ast(T)=0 &\text{ in }\Omega,
			\end{cases}
		\end{align}
		In other words, given a nonempty open set $W\subset \Omega_e$, the set 
		$$
		\mathcal{R}^{\ast}\vcentcolon = \left\{v^\ast_g -g ;\, g\in C^\infty_c(W_T)\right\}
		$$ 
		is dense in $L^2(0,T\,;\widetilde{H}^s(\Omega))$.
	\end{itemize}
\end{remark}

\subsection{Interior determination and proof of Theorem~\ref{Theorem: General formulation}}\label{sec: interior det}

Let us state the interior uniqueness result.

\begin{theorem}[Interior uniqueness]\label{Thm: interior uniqueness}
	Let $\Omega\subset \R^n$ be an open set bounded in one direction, $0<T<\infty$, $0<s<\min(1,n/2)$, $\gamma_0>0$ and $\gamma_j\in \Gamma_{s,\gamma_0}(\R^n_T)$ for $j=1,2$. Assume that $W_1,W_2\subset \Omega_e$ are two disjoint nonempty open sets and $\Gamma\in \Gamma_{s,\gamma_0}(\R^n_T)$ are such that $\gamma_1(x,t)=\gamma_2(x,t)=\Gamma(x,t)$ for all $(x,t)\in (W_1\cup W_2)_T$ and $\Gamma\in C^{\infty}((W_1\cup W_2)_T)$. Then 
	\begin{align}\label{equal DN inter unique}
		\langle\mathcal{N}_{\gamma_1}f,g\rangle =\langle
		\mathcal{N}_{\gamma_2}f,g\rangle
	\end{align}
	if and only if
	\[
	\gamma_1=\gamma_2  \text{ and }Q_{\gamma_1}=Q_{\gamma_2} \text{ in }\Omega_T.
	\]
\end{theorem}

\begin{proof}
	Via Theorem \ref{Thm: relation of DNs}, let $g\in C^\infty_c((W_2)_T)$, then one has 
	\[
	\langle\mathcal{N}_{\gamma_1}f,g\rangle =\langle
	\mathcal{N}_{\gamma_2}f,g\rangle
	\]
	if and only if
	\begin{equation}
		\label{eq: useful rel interior uniqueness}
		\langle\mathcal{N}_{Q_{\gamma_1}}(\Gamma^{1/2}f),(\Gamma^{1/2}g)\rangle = \langle\mathcal{N}_{Q_{\gamma_2}}(\Gamma^{1/2}f),(\Gamma^{1/2}g)\rangle.
	\end{equation}
	Since $\Gamma$ is uniformly elliptic and smooth on $(W_1\cup W_2)_T$, the condition \eqref{eq: useful rel interior uniqueness} implies
	\[
	\langle\mathcal{N}_{Q_{\gamma_1}}f,g\rangle = \langle\mathcal{N}_{Q_{\gamma_2}}f,g\rangle
	\]
	for all $f\in C_c^{\infty}((W_1)_T)$ and $g\in C_c^{\infty}((W_2)_T)$. Moreover, by Proposition \ref{prop: integral identity}, one has 
	\begin{align}\label{integral id =0 in pf}
		\int_{\Omega_T}(\gamma_2^{-1}-\gamma_1^{-1})v_f\partial_tv^\ast _g\,dxdt+\int_{\Omega_T}(Q_{\gamma_1}-Q_{\gamma_2})v^\ast_gv_f\,dxdt=0,
	\end{align}
	where 
	$v_{f}\in \mathcal{H}\vcentcolon = H^1(0,T\,;L^2(\Omega))\cap L^2(0,T\,;H^S(\R^n))$ and $v_{g}^\ast \in \mathcal{H}$ are the solutions to 
	\begin{align}\label{sols sequ runge}
		\begin{cases}
			\partial_t\LC \gamma_1^{-1}v_{f}\RC +\LC (-\Delta)^s+Q_{\gamma_1}\RC v_{f}=0 & \text{ in } \Omega_T,\\
			v_{f}=f &\text{ in }(\Omega_e)_T,\\
			v_{f}(0)=0 &\text{ in }\Omega,
		\end{cases}
	\end{align}
	and
	\begin{align}\label{sols sequ runge adj}
		\begin{cases}
			-\gamma_2^{-1}\partial_tv_g^\ast  +\LC (-\Delta)^s+Q_{\gamma_2}\RC v_g^\ast =0 & \text{ in } \Omega_T,\\
			v_g^\ast=g &\text{ in }(\Omega_e)_T,\\
			v_g^\ast(T)=0 &\text{ in }\Omega,
		\end{cases}
	\end{align}
	respectively.
	
	\medskip
	
	{\it Step 1. $Q_{\gamma_1}=Q_{\gamma_2}$ in $\Omega_T$.}
	
	\medskip
	
	\noindent Take $\Omega'\Subset \Omega$ and $\psi \in C_c^{\infty}(\Omega)$ with $\psi|_{\overline{\Omega'}}=1$. By extending $\psi$ for all times trivially, we have $\psi \in \mathcal{H}$. Then take $\phi\in C_c^{\infty}(\Omega'_T)$ and apply the Runge approximation (Proposition \ref{Prop: runge} and Remark \ref{Rmk: runge}) to find  sequences $\left\{f_\ell\right\}_{\ell=1}^\infty\subset C^\infty_c((W_1)_T)$ and $\left\{ g_k\right\}_{k=1}^{\infty}\subset C^\infty_c((W_2)_T)$ such that $v_{f_{\ell}}-f_{\ell}\to \phi$ and $v_{g_k}^{\ast}-g_k\to \psi$ as $\ell,k\to \infty$. 
	Here $v_{f_\ell}\in \mathcal{H}$ and $v_{g_k}^{\ast}\in \mathcal{H}$ are the solutions to \eqref{sols sequ runge} and \eqref{sols sequ runge adj} with $f=f_\ell$ and $g=g_k$, respectively. Hence,
	\[
	\lim_{\ell, k\to\infty}\int_{\Omega_T}(\gamma_2^{-1}-\gamma_1^{-1})v_{f_\ell}\partial_tv^\ast_{g_k}\,dxdt=0
	\]
	and 
	\[
	\lim_{\ell,k\to\infty}\int_{\Omega_T}(Q_{\gamma_1}-Q_{\gamma_2})v^\ast _{g_k}v_{f_{\ell}}\,dxdt=\int_{\Omega_T}(Q_{\gamma_1}-Q_{\gamma_2})\psi \phi\,dxdt.
	\]
	Using that $\phi\psi=\phi$, we deduce 
	\[
	\int_{\Omega_T}(Q_{\gamma_1}-Q_{\gamma_2}) \phi\,dxdt=0,
	\]
	for any possible $\phi\in C^{\infty}_c(\Omega_T)$. Thus, one can conclude that $Q_{\gamma_1}=Q_{\gamma_2}$ in $\Omega_T$.
	
	\medskip
	
	{\it Step 2. $\gamma_1=\gamma_2$ in $\Omega_T$.}
	
	\medskip
	
	\noindent  Plugging $Q_{\gamma_1}=Q_{\gamma_2}$ in $\Omega_T$ into \eqref{integral id =0 in pf}, we have 
	\begin{align}
		\int_{\Omega_T}(\gamma_2^{-1}-\gamma_1^{-1})v_f\partial_tv^\ast _g\,dxdt=0, 
	\end{align}
	for any $f\in C^\infty_c((W_1)_T)$ and $g \in C^\infty_c ((W_2)_T)$. Take $\Omega'\Subset \Omega$, $\eta \in C_c^{\infty}(\Omega)$ with $\eta|_{\overline{\Omega'}}=1$ and set $\psi(\cdot, t)=t\eta$. By repeating the arguments as in {\it Step 1}, with the Runge approximation at hand, then one can also conclude that 
	\begin{align}
		\int_{\Omega_T}(\gamma_2^{-1}-\gamma_1^{-1})\phi \,dxdt=0, 
	\end{align}
	for any possible $\phi\in C_c^{\infty}(\Omega_T)$. This ensures $\gamma_1=\gamma_2$ in $\Omega_T$. 
\end{proof}

\begin{proof}[Proof of Theorem~\ref{Theorem: General formulation}]
    First we apply Theorem~\ref{Theorem: Exterior determination} to deduce that $\gamma_1=\gamma_2$ in $W_T$. Then we choose two nonempty disjoint open sets $W_1,W_2\subset W$. By Lemma~\ref{Lemma: old DN implies new DN} the condition \eqref{same DN in thm 1} implies that the identity \eqref{equal DN inter unique} holds for $W_1,W_2$ as chosen initially. Now by using Theorem~\ref{Thm: interior uniqueness} we can conclude that $\gamma_1=\gamma_2$, $Q_{\gamma_1}=Q_{\gamma_2}$ in $\Omega_T$. This in turn implies
    \[  
	\begin{split}
		0&=Q_{\gamma_1}-Q_{\gamma_2}=(-\Delta)^s(m_{\gamma_2}-m_{\gamma_1}) \text{ in }\Omega,
	\end{split}
	\]
	for a.e. $t\in (0,T)$. Hence, by the UCP (see~\cite[Theorem~2.2]{KRZ2022Biharm}) it follows that $\gamma_1=\gamma_2$ in $\R^n_T$.
\end{proof}

\subsection*{Acknowledgments} Y.-H.L. was partially supported by MOST 111-2628-M-A49-002. J.R. was supported by the Vilho, Yrj\"o and Kalle V\"ais\"al\"a Foundation of the Finnish Academy of Science and Letters.

\appendix

\section{Discussion of nonlocal normal derivatives and DN maps}
\label{sec: Discussion of nonlocal normal derivatives and DN maps}

In this section, we motivate the definition of the nonlocal Neumann derivatives $\mathcal{N}_{\gamma}$ and $\mathcal{N}_{Q_{\gamma}}$, which underly the Definition~\ref{def: DN maps} and \ref{def: DN map schroedinger type equation}. We restrict here our attention to time independent functions for simplicity. First recall that from \cite[Lemma~3.3]{XavierNeumannBdry}, one has the nonlocal integration by parts formula
\begin{equation}
	\label{eq: Gauss theorem fractional Laplacian}
	\begin{split}
		&\int_{\Omega_e}(\mathcal{N}_su)v\,dx+\int_{\Omega}((-\Delta)^su)v\,dx \\
		=&\frac{C_{n,s}}{2}\int_{\R^{2n}\setminus(\Omega_e\times\Omega_e)}\frac{(u(x)-u(y))(v(x)-v(y))}{|x-y|^{n+2s}}\,dxdy
	\end{split}
\end{equation}
for all $u,v\in C^2(\R^n)$, where 
\begin{equation}
	\label{eq: nonlocal normal derivative}
	\mathcal{N}_su(x)\vcentcolon = C_{n,s}\int_{\Omega}\frac{u(x)-u(y)}{|x-y|^{n+2s}}\,dy
\end{equation}
denotes the nonlocal normal derivative for $x\in \Omega_e$ and sufficiently regular functions $u\colon\R^n\to\R$. Here $C_{n,s}$ is the same constant given by \eqref{C_ns}.

Next, we want to show that a similar formula holds for the fractional conductivity operator studied in this work. For simplicity assume $u,\phi\in C_c^{\infty}(\R^n)$ and denote the duality pairing between $H^s(\R^n)$ and $H^{-s}(\R^n)$ by $\langle\cdot,\cdot\rangle$. Then we have
\begin{equation}
	\label{Intro: eq: strong formulation}
	\begin{split}
		\langle \Div_s(\Theta_{\gamma}\nabla^su),\phi\rangle=&\frac{C_{n,s}}{2}\int_{\R^n}\int_{\R^n}\frac{\gamma^{1/2}(x)\gamma^{1/2}(y)}{|x-y|^{n+2s}}(u(x)-u(y))(\phi(x)-\phi(y))\,dxdy\\
		=&\frac{C_{n,s}}{2}\,\int_{\R^n}\left(\int_{\R^n}\frac{\gamma^{1/2}(y)}{|x-y|^{n+2s}}(u(x)-u(y))\,dy\right)\gamma^{1/2}(x)\phi(x)\,dx\\
		&-\frac{C_{n,s}}{2}\,\int_{\R^n}\left(\int_{\R^n}\frac{\gamma^{1/2}(x)}{|x-y|^{n+2s}}(u(x)-u(y))\,dx\right)\gamma^{1/2}(y)\phi(y)\,dy\\
		=&C_{n,s}\,\int_{\R^n}\left(\int_{\R^n}\frac{\gamma^{1/2}(y)}{|x-y|^{n+2s}}(u(x)-u(y))\,dy\right)\gamma^{1/2}(x)\phi(x)\,dx\\
		=&\int_{\R^n}L_{\gamma}u(x)\phi(x)\,dx,
	\end{split}
\end{equation}
where we set
\begin{equation}
	\label{Intro: eq: Strong conductivity operator 1}
	L_{\gamma}^su(x)\vcentcolon =C_{n,s}\gamma^{1/2}(x)\int_{\R^n}\frac{\gamma^{1/2}(y)}{|x-y|^{n+2s}}(u(x)-u(y))\,dy.
\end{equation}

Now, let us define the (general) nonlocal Neumann derivative by
\begin{equation}
	\label{eq: neumann derivative conductivity eq}
	\mathcal{N}_s^{\gamma}u(x)\vcentcolon = C_{n,s}\gamma^{1/2}(x)\int_{\Omega}\gamma^{1/2}(y)\frac{u(x)-u(y)}{|x-y|^{n+2s}}\,dy
\end{equation}
for $x\in \Omega_e$ and sufficiently regular functions $u\colon \R^n\to\R$ (in this section we write the superscript $\gamma$ to distinguish it from the normal derivative in \eqref{eq: nonlocal normal derivative}). With this definition we have
\begin{equation}
	\label{eq: Gauss theorem fractional conductivity equation}
	\begin{split}
		&\frac{C_{n,s}}{2}\int_{\R^{2n}\setminus(\Omega_e\times\Omega_e)}\gamma^{1/2}(x)\gamma^{1/2}(y)\frac{(u(x)-u(y))(v(x)-v(y))}{|x-y|^{n+2s}}\,dxdy\\
		=&\int_{\Omega}L_{\gamma}^su(x) v(x)\,dx+\int_{\Omega_e}v(x)\mathcal{N}_s^{\gamma}u(x)\,dx.
	\end{split}
\end{equation}
To see this, observe that $\R^{2n}\setminus(\Omega_e\times\Omega_e)=(\Omega\times \R^n)\cup (\Omega_e\times\Omega)$ and hence
\[
\begin{split}
	&\int_{\R^{2n}\setminus(\Omega_e\times\Omega_e)}\gamma^{1/2}(x)\gamma^{1/2}(y)\frac{(u(x)-u(y))(v(x)-v(y))}{|x-y|^{n+2s}}\,dxdy\\
	=&\int_{\R^{2n}\setminus(\Omega_e\times\Omega_e)}\gamma^{1/2}(x)\gamma^{1/2}(y)\frac{u(x)-u(y)}{|x-y|^{n+2s}}v(x)\,dxdy\\
	&-\int_{\R^{2n}\setminus(\Omega_e\times\Omega_e)}\gamma^{1/2}(x)\gamma^{1/2}(y)\frac{u(x)-u(y)}{|x-y|^{n+2s}}v(y)\,dxdy\\
	=&2\int_{\R^{2n}\setminus(\Omega_e\times\Omega_e)}\gamma^{1/2}(x)\gamma^{1/2}(y)\frac{u(x)-u(y)}{|x-y|^{n+2s}}v(x)\,dxdy\\
	=&2\int_{\Omega}v(x)\left(\int_{\R^n}\gamma^{1/2}(x)\gamma^{1/2}(y)\frac{u(x)-u(y)}{|x-y|^{n+2s}}\,dy\right)dx\\
	&+ 2\int_{\Omega_e}v(x)\left(\int_{\Omega}\gamma^{1/2}(x)\gamma^{1/2}(y)\frac{u(x)-u(y)}{|x-y|^{n+2s}}\,dy\right)\,dx.
\end{split}
\]
By \eqref{Intro: eq: Strong conductivity operator 1} and \eqref{eq: neumann derivative conductivity eq} this implies the identity \eqref{eq: Gauss theorem fractional conductivity equation}. From the identity \eqref{eq: Gauss theorem fractional conductivity equation} we make the following observations:

\begin{enumerate}[(i)]
	\item\label{property 1}  There holds
	\begin{equation}
		\label{eq: testing with functions supported in exterior}
		\begin{split}
			&\frac{C_{n,s}}{2}\int_{\R^{2n}\setminus(\Omega_e\times\Omega_e)}\gamma^{1/2}(x)\gamma^{1/2}(y)\frac{(u(x)-u(y))(v(x)-v(y))}{|x-y|^{n+2s}}\,dxdy\\
			=&\int_{\Omega_e}v(x)\mathcal{N}_s^{\gamma}u(x)\,dx.
		\end{split}
	\end{equation}
	for all $v\in C_c^{\infty}(\Omega_e)$ and so coincides with our weak formulation.
	\item\label{property 2} We have
	\begin{equation}
		\label{eq: test functions}
		\begin{split}
			&\frac{C_{n,s}}{2}\int_{\R^{2n}\setminus(\Omega_e\times\Omega_e)}\gamma^{1/2}(x)\gamma^{1/2}(y)\frac{(u(x)-u(y))(v(x)-v(y))}{|x-y|^{n+2s}}\,dxdy\\
			=&\int_{\Omega}L_{\gamma}^su(x) v(x)\,dx\\
			=&\frac{C_{n,s}}{2}\int_{\R^{2n}}\frac{\gamma^{1/2}(x)\gamma^{1/2}(y)}{|x-y|^{n+2s}}(u(x)-u(y))(v(x)-v(y))\,\,dxdy
		\end{split}
	\end{equation}
	for all $v\in C_c^{\infty}(\Omega)$.
\end{enumerate}
Note that all observations above hold for a general symmetric kernel $K(x,y)$. Combining the assertion \ref{property 1} and \ref{property 2}, we see that:
\begin{enumerate}[(a)]
	\item The notion of solutions in the survey article for elliptic nonlocal equations in \cite{ros2016nonlocal} and the definitions adapted in this article are the same. The former one have the advantage that one can study solutions to nonlocal Dirichlet problems, where the exterior conditions $f$ are less regular.
	\item We have 
	\[
	\langle \mathcal{N}_s^{\gamma}f,g\rangle =\langle \mathcal{N}_s^{\gamma}f,g'\rangle,
	\]
	whenever $g,g'\in H^s(\R^n)$ satisfy $g-g'\in\widetilde{H}^s(\Omega)$, where $\mathcal{N}_sf$ is the nonlocal normal derivative of the unique solution $u_f$ to the homogeneous fractional conductivity equation with exterior value $f$. Hence, its again well-defined on the trace space $X=H^s(\R^n)/\widetilde{H}^s(\Omega)$.
\end{enumerate}

Moreover, let us point out that in \cite{RZ2022unboundedFracCald,RZ2022LowReg} we used the following definition of DN map
\begin{equation}
	\label{eq: DN map elliptic}
	\langle \Lambda_{\gamma}f,g\rangle=\frac{C_{n,s}}{2}\int_{\R^{2n}}\gamma^{1/2}(x)\gamma^{1/2}(y)\frac{(u_f(x)-u_f(y))(g(x)-g(y))}{|x-y|^{n+2s}}\,dxdy
\end{equation}
for all $f,g\in C_c^{\infty}(\Omega_e)$. These two are related as follows 
\begin{equation}
	\label{eq: relation different DN maps}
	\begin{split}
		\langle \Lambda_{\gamma}f,g\rangle =&\langle\mathcal{N}_s^{\gamma}f,g\rangle+\frac{C_{n,s}}{2}\int_{\Omega_e\times\Omega_e}\gamma^{1/2}(x)\gamma^{1/2}(y)\frac{(u_f(x)-u_f(y))(g(x)-g(y))}{|x-y|^{n+2s}}\,dxdy\\
		=&\langle\mathcal{N}_s^{\gamma}f,g\rangle+\frac{C_{n,s}}{2}\int_{\Omega_e\times\Omega_e}\gamma^{1/2}(x)\gamma^{1/2}(y)\frac{(f(x)-f(y))(g(x)-g(y))}{|x-y|^{n+2s}}\,dxdy.
	\end{split}
\end{equation}
\begin{remark}
	We may now observe that the additional information on the set $\Omega_e\times\Omega_e$ precisely allows to carry out the exterior determination. This shows that $\Lambda_{\gamma}$ carries more information.
\end{remark}
As a matter of fact, the definition $\mathcal{N}_{\gamma}^s$ is natural since it has a clear PDE interpretation although we cannot prove with it our exterior determination result. 


Finally, we discuss the situation for constant coefficient operators like the fractional Schr\"odinger equation
\begin{equation}
	\label{eq: Schroedinger operator}
	\begin{cases}
		\LC (-\Delta)^s+q\RC u=0 &\text{ in }\Omega,\\
		u=f  &\text{ in }\Omega_e.
	\end{cases}
\end{equation}
In \cite{GSU20}, the authors defined the DN map $\Lambda_q$ related to this exterior value problem by
\[
\langle\Lambda_qf,g\rangle=\int_{\R^n}(-\Delta)^{s/2}u_f(-\Delta)^{s/2}v_g\,dx+\int_{\Omega}qu_fv_g\,dx
\]
for all $f,g\in H^s(\R^n)/ \widetilde{H}^s(\Omega)$, where $u_f\in H^s(\R^n)$ is the weak solution to \eqref{eq: Schroedinger operator} and $v_g\in H^s(\R^n)$ an extension of $g$. In the special case $g\in C_c^{\infty}(\Omega_e)$ one has
\[
\langle\Lambda_qf,g\rangle=\int_{\R^n}(-\Delta)^{s/2}u_f(-\Delta)^{s/2}g\,dx,
\]
since we are only integrating over $\Omega$ in the potential and so $q$ is only implicitly contained in the definition of $\Lambda_q$. Then they showed in \cite[Lemma~3.1]{GSU20} that if $\Omega\Subset\R^n$ is smooth and $q\in C_c^{\infty}(\Omega)$ this DN map is simply the restriction $\left.(-\Delta)^su_f\right|_{\Omega_e}$ (as long as the data $f,g$ are sufficiently regular) and in the case $f\in C_c^{\infty}(\Omega_e)$ there holds
\begin{equation}
	\Lambda_qf=\mathcal{N}_sf-mf+\left.(-\Delta)^sf\right|_{\Omega_e},
\end{equation}
where $m(x)=C_{n,s}\int_{\Omega}\frac{dy}{|x-y|^{n+2s}}$ (cf.~\cite[Lemma~A.2]{GSU20}). But this implies in this case that
\[
\Lambda_{q_1}=\Lambda_{q_2}\quad\Longleftrightarrow\quad  \mathcal{N}_s^1=\mathcal{N}_s^2.
\]
If $q$ is possibly nontrivial in the exterior then the notion of solutions to \eqref{eq: Schroedinger operator} is not affected if one introduces the related bilinear form by
\begin{equation}
	\label{eq: our definition}
	B_q(u,v)\vcentcolon = \int_{\R^n}(-\Delta)^{s/2}u(-\Delta)^{s/2}v\,dx+\int_{\R^n}quv\,dx
\end{equation}
for $u,v\in H^s(\R^n)$. This approach was, for example, carried out in \cite{RZ2022unboundedFracCald} or \cite{RS-fractional-calderon-low-regularity-stability}. But then the natural DN map becomes 
\begin{equation}
	\label{eq: new DN map}
	\langle \widetilde{\Lambda}_qf,g\rangle\vcentcolon =\int_{\R^n}(-\Delta)^{s/2}u_f(-\Delta)^{s/2}v_g\,dx+\int_{\R^n}qu_fv_g\,dx
\end{equation}
for all $f,g\in H^s(\R^n)/\widetilde{H}^s(\Omega)$, where $v_g$ is any representative of $g$. These two definitions of DN maps are related as follows
\begin{equation}
	\label{eq: new rel of DN maps}
	\begin{split}
		\langle \widetilde{\Lambda}_qf,g\rangle=&\langle \Lambda_qf,g\rangle+\int_{\Omega_e}qu_fv_g\,dx\\
		=&\langle \Lambda_qf,g\rangle+\int_{\Omega_e}qfg\,dx.
	\end{split}
\end{equation}
Hence, in general if $q$ is not zero in the exterior these two definitions of DN maps are not equivalent and the latter helps to acquire information in the exterior. Therefore if \eqref{eq: new rel of DN maps} $f,g$ have disjoint support then they are equivalent and precisely this lack of knowledge leads to counterexamples to uniqueness (cf.~\cite{RZ2022FracCondCounter}).

\medskip

\bibliography{refs} 

\begin{thebibliography}{CMRU22}

\bibitem[AF92]{AdamsComposition}
David~R. Adams and Michael Frazier.
\newblock Composition operators on potential spaces.
\newblock {\em Proc. Amer. Math. Soc.}, 114(1):155--165, 1992.

\bibitem[AP06]{AP06}
Kari Astala and Lassi P\"{a}iv\"{a}rinta.
\newblock Calder\'{o}n's inverse conductivity problem in the plane.
\newblock {\em Ann. of Math. (2)}, 163(1):265--299, 2006.

\bibitem[BGU21]{bhattacharyya2021inverse}
S.~Bhattacharyya, T.~Ghosh, and G.~Uhlmann.
\newblock Inverse problems for the fractional-{L}aplacian with lower order
  non-local perturbations.
\newblock {\em Trans. Amer. Math. Soc.}, 374(5):3053--3075, 2021.

\bibitem[Bre11]{Brezis}
Haim Brezis.
\newblock {\em Functional analysis, {S}obolev spaces and partial differential
  equations}.
\newblock Universitext. Springer, New York, 2011.

\bibitem[Cal06]{calderon2006inverse}
Alberto~P Calder{\'o}n.
\newblock On an inverse boundary value problem.
\newblock {\em Computational \& Applied Mathematics}, 25(2-3):133--138, 2006.

\bibitem[CK01]{canuto2001determining}
B.~Canuto and O.~Kavian.
\newblock Determining coefficients in a class of heat equations via boundary
  measurements.
\newblock {\em SIAM J. Math. Anal.}, 32(5):963--986, 2001.

\bibitem[CL19]{CL2019determining}
Xinlin Cao and Hongyu Liu.
\newblock Determining a fractional {H}elmholtz equation with unknown source and
  scattering potential.
\newblock {\em Commun. Math. Sci.}, 17(7):1861--1876, 2019.

\bibitem[CLL19]{CLL2017simultaneously}
Xinlin Cao, Yi-Hsuan Lin, and Hongyu Liu.
\newblock Simultaneously recovering potentials and embedded obstacles for
  anisotropic fractional {S}chr\"{o}dinger operators.
\newblock {\em Inverse Probl. Imaging}, 13(1):197--210, 2019.

\bibitem[CLR20]{cekic2020calderon}
Mihajlo Cekic, Yi-Hsuan Lin, and Angkana R{\"u}land.
\newblock The {C}alder{\'o}n problem for the fractional {S}chr{\"o}dinger
  equation with drift.
\newblock {\em Cal. Var. Partial Differential Equations}, 59(91), 2020.

\bibitem[CMR21]{CMR20}
Giovanni Covi, Keijo M\"{o}nkk\"{o}nen, and Jesse Railo.
\newblock Unique continuation property and {P}oincar\'{e} inequality for higher
  order fractional {L}aplacians with applications in inverse problems.
\newblock {\em Inverse Probl. Imaging}, 15(4):641--681, 2021.

\bibitem[CMRU22]{CMRU20}
Giovanni Covi, Keijo M\"{o}nkk\"{o}nen, Jesse Railo, and Gunther Uhlmann.
\newblock The higher order fractional {C}alder\'{o}n problem for linear local
  operators: {U}niqueness.
\newblock {\em Adv. Math.}, 399:Paper No. 108246, 2022.

\bibitem[Cov20]{C20}
Giovanni Covi.
\newblock Inverse problems for a fractional conductivity equation.
\newblock {\em Nonlinear Anal.}, 193:111418, 18, 2020.

\bibitem[CRTZ22]{StabilityFracCond}
Giovanni Covi, Jesse Railo, Teemu Tyni, and Philipp Zimmermann.
\newblock Stability estimates for the inverse fractional conductivity problem.
\newblock {\em arXiv:2210.01875}, 2022.

\bibitem[CRZ22]{CRZ2022global}
Giovanni Covi, Jesse Railo, and Philipp Zimmermann.
\newblock The global inverse fractional conductivity problem.
\newblock {\em arXiv:2204.04325}, 2022.

\bibitem[DGLZ12]{DGLZ12}
Qiang Du, Max Gunzburger, R.B. Lehoucq, and Kun Zhou.
\newblock Analysis and approximation of nonlocal diffusion problems with volume
  constraints.
\newblock {\em SIAM rev 54, No 4:667-696}, 2012.

\bibitem[DL92]{DautrayLionsVol5}
Robert Dautray and Jacques-Louis Lions.
\newblock {\em Mathematical analysis and numerical methods for science and
  technology. {V}ol. 5}.
\newblock Springer-Verlag, Berlin, 1992.
\newblock Evolution problems. I, With the collaboration of Michel Artola,
  Michel Cessenat and H\'{e}l\`ene Lanchon, Translated from the French by Alan
  Craig.

\bibitem[DROV17]{XavierNeumannBdry}
Serena Dipierro, Xavier Ros-Oton, and Enrico Valdinoci.
\newblock Nonlocal problems with {N}eumann boundary conditions.
\newblock {\em Rev. Mat. Iberoam.}, 33(2):377--416, 2017.

\bibitem[Eva10]{EvansPDE}
Lawrence~C. Evans.
\newblock {\em Partial differential equations}, volume~19 of {\em Graduate
  Studies in Mathematics}.
\newblock American Mathematical Society, Providence, RI, second edition, 2010.

\bibitem[Fei22]{AliParCald}
Ali Feizmohammadi.
\newblock An inverse boundary value problem for isotropic nonautonomous heat
  flows.
\newblock {\em arXiv:2203.13742}, 2022.

\bibitem[FGKU21]{feizmohammadi2021fractional}
Ali Feizmohammadi, Tuhin Ghosh, Katya Krupchyk, and Gunther Uhlmann.
\newblock Fractional anisotropic {C}alder\'on problem on closed {R}iemannian
  manifolds.
\newblock {\em arXiv:2112.03480}, 2021.

\bibitem[Gho21]{ghosh2021non}
Tuhin Ghosh.
\newblock A non-local inverse problem with boundary response.
\newblock {\em Rev. Mat. Iberoam.}, 2021.

\bibitem[GLX17]{GLX}
Tuhin Ghosh, Yi-Hsuan Lin, and Jingni Xiao.
\newblock The {C}alder\'{o}n problem for variable coefficients nonlocal
  elliptic operators.
\newblock {\em Comm. Partial Differential Equations}, 42(12):1923--1961, 2017.

\bibitem[GRSU20]{GRSU18}
Tuhin Ghosh, Angkana R\"{u}land, Mikko Salo, and Gunther Uhlmann.
\newblock Uniqueness and reconstruction for the fractional {C}alder\'{o}n
  problem with a single measurement.
\newblock {\em J. Funct. Anal.}, 279(1):108505, 42, 2020.

\bibitem[GSU20]{GSU20}
Tuhin Ghosh, Mikko Salo, and Gunther Uhlmann.
\newblock The {C}alder\'{o}n problem for the fractional {S}chr\"{o}dinger
  equation.
\newblock {\em Anal. PDE}, 13(2):455--475, 2020.

\bibitem[GU21]{GU2021calder}
Tuhin Ghosh and Gunther Uhlmann.
\newblock The {C}alder\'{o}n problem for nonlocal operators.
\newblock {\em arXiv:2110.09265}, 2021.

\bibitem[HL19]{harrach2017nonlocal-monotonicity}
Bastian Harrach and Yi-Hsuan Lin.
\newblock Monotonicity-based inversion of the fractional {S}chr\"{o}dinger
  equation {I}. {P}ositive potentials.
\newblock {\em SIAM J. Math. Anal.}, 51(4):3092--3111, 2019.

\bibitem[HL20]{harrach2020monotonicity}
Bastian Harrach and Yi-Hsuan Lin.
\newblock Monotonicity-based inversion of the fractional {S}ch\"{o}dinger
  equation {II}. {G}eneral potentials and stability.
\newblock {\em SIAM J. Math. Anal.}, 52(1):402--436, 2020.

\bibitem[KLW22]{KLW2021calder}
Pu-Zhao Kow, Yi-Hsuan Lin, and Jenn-Nan Wang.
\newblock The {C}alder\'{o}n problem for the fractional wave equation:
  uniqueness and optimal stability.
\newblock {\em SIAM J. Math. Anal.}, 54(3):3379--3419, 2022.

\bibitem[KOSY18]{GlobalTimeFrac}
Y.~Kian, L.~Oksanen, E.~Soccorsi, and M.~Yamamoto.
\newblock Global uniqueness in an inverse problem for time fractional diffusion
  equations.
\newblock {\em J. Differential Equations}, 264(2):1146--1170, 2018.

\bibitem[KRZ22]{KRZ2022Biharm}
Manas Kar, Jesse Railo, and Philipp Zimmermann.
\newblock The fractional $p$-biharmonic systems: optimal {P}oincar\'e
  constants, unique continuation and inverse problems.
\newblock {\em arXiv:2208.09528}, 2022.

\bibitem[KSY18]{Time-Frac-Diff}
Yavar Kian, Eric Soccorsi, and Masahiro Yamamoto.
\newblock On time-fractional diffusion equations with space-dependent variable
  order.
\newblock {\em Ann. Henri Poincar\'{e}}, 19(12):3855--3881, 2018.

\bibitem[Lad85]{LadyzhenskajaBVP}
O.~A. Ladyzhenskaya.
\newblock {\em The boundary value problems of mathematical physics}, volume~49
  of {\em Applied Mathematical Sciences}.
\newblock Springer-Verlag, New York, 1985.
\newblock Translated from the Russian by Jack Lohwater [Arthur J. Lohwater].

\bibitem[Lin22]{lin2020monotonicity}
Yi-Hsuan Lin.
\newblock Monotonicity-based inversion of fractional semilinear elliptic
  equations with power type nonlinearities.
\newblock {\em Calc. Var. Partial Differential Equations}, 61(5):Paper No. 188,
  30, 2022.

\bibitem[LL22a]{LL2020inverse}
Ru-Yu Lai and Yi-Hsuan Lin.
\newblock Inverse problems for fractional semilinear elliptic equations.
\newblock {\em Nonlinear Anal.}, 216:Paper No. 112699, 21, 2022.

\bibitem[LL22b]{LL2022inverse}
Yi-Hsuan Lin and Hongyu Liu.
\newblock Inverse problems for fractional equations with a minimal number of
  measurements.
\newblock {\em arXiv:2203.03010}, 2022.

\bibitem[LLR20]{LLR2019calder}
Ru-Yu Lai, Yi-Hsuan Lin, and Angkana R\"{u}land.
\newblock The {C}alder\'{o}n problem for a space-time fractional parabolic
  equation.
\newblock {\em SIAM J. Math. Anal.}, 52(3):2655--2688, 2020.

\bibitem[LLU22]{LLU2022calder}
Ching-Lung Lin, Yi-Hsuan Lin, and Gunther Uhlmann.
\newblock The {C}alder\'on problem for nonlocal parabolic operators.
\newblock {\em arXiv:2209.11157}, 2022.

\bibitem[LSU88]{LadyzhenskajaParabolic}
Ol'ga~Aleksandrovna Lady{\v{z}}enskaja, Vsevolod~Alekseevich Solonnikov, and
  Nina~N Ural'tseva.
\newblock {\em Linear and quasi-linear equations of parabolic type}, volume~23.
\newblock American Mathematical Soc., 1988.

\bibitem[Nac96]{Nachman1996GlobalUniqueness}
Adrian~I. Nachman.
\newblock Global uniqueness for a two-dimensional inverse boundary value
  problem.
\newblock {\em Ann. of Math. (2)}, 143(1):71--96, 1996.

\bibitem[RO16]{ros2016nonlocal}
Xavier Ros-Oton.
\newblock Nonlocal elliptic equations in bounded domains: a survey.
\newblock {\em Publicacions matematiques}, pages 3--26, 2016.

\bibitem[RS18]{ruland2018exponential}
Angkana R{\"u}land and Mikko Salo.
\newblock Exponential instability in the fractional {C}alder{\'o}n problem.
\newblock {\em Inverse Problems}, 34(4):045003, 2018.

\bibitem[RS20a]{RS17}
Angkana R\"{u}land and Mikko Salo.
\newblock The fractional {C}alder\'{o}n problem: low regularity and stability.
\newblock {\em Nonlinear Anal.}, 193:111529, 56, 2020.

\bibitem[RS20b]{RS-fractional-calderon-low-regularity-stability}
Angkana R\"{u}land and Mikko Salo.
\newblock The fractional {C}alder\'{o}n problem: low regularity and stability.
\newblock {\em Nonlinear Anal.}, 193:111529, 56, 2020.

\bibitem[RZ22a]{RZ2022FracCondCounter}
Jesse Railo and Philipp Zimmermann.
\newblock Counterexamples to uniqueness in the inverse fractional conductivity
  problem with partial data.
\newblock {\em Inverse Probl. Imaging (to appear)}, 2022.
\newblock arXiv:2203.02442.

\bibitem[RZ22b]{RZ2022unboundedFracCald}
Jesse Railo and Philipp Zimmermann.
\newblock Fractional {C}alderón problems and {P}oincaré inequalities on
  unbounded domains.
\newblock {\em J. Spectr. Theory (to appear)}, 2022.
\newblock arXiv:2203.02425.

\bibitem[RZ22c]{RZ2022LowReg}
Jesse Railo and Philipp Zimmermann.
\newblock Low regularity theory for the inverse fractional conductivity
  problem.
\newblock {\em arXiv:2208.11465}, 2022.

\bibitem[SU87]{SU87}
John Sylvester and Gunther Uhlmann.
\newblock A global uniqueness theorem for an inverse boundary value problem.
\newblock {\em Ann. of Math. (2)}, 125(1):153--169, 1987.

\end{thebibliography}

\bibliographystyle{alpha}

\end{document}